\documentclass[12pt]{extarticle}
\usepackage{amsmath, amsxtra, amsthm, amsfonts, amssymb, color}
\usepackage[colorlinks=true,linkcolor=blue,urlcolor=blue]{hyperref}
\usepackage{graphicx}
\usepackage{caption}
\usepackage{mathtools}
\usepackage{enumerate}
\usepackage{verbatim}
\usepackage{tikz,tikz-cd,tikz-3dplot}
\usepackage{amssymb}
\usetikzlibrary{matrix}
\usetikzlibrary{arrows}
\usepackage{algorithm}
\usepackage[noend]{algpseudocode}
\usepackage{caption}
\usepackage[normalem]{ulem}
\usepackage{subcaption}
\usepackage{multicol}
\oddsidemargin \evensidemargin
\usepackage{makecell}
\usepackage{array}
\usepackage{enumitem}

 \theoremstyle{theorem}
 \newtheorem{theorem}{Theorem}[section]

\newtheorem{proposition}[theorem]{Proposition}
\newtheorem{lemma}[theorem]{Lemma}
\newtheorem{corollary}[theorem]{Corollary}
\theoremstyle{definition}
\newtheorem{algo}[theorem]{Algorithm}

\newtheorem{remark}[theorem]{Remark}

\newtheorem{examplex}[theorem]{Example}
\newenvironment{example}
{\pushQED{\qed}\examplex}
{\popQED\endexamplex}

\newtheorem{question}[theorem]{Question}

\newcommand{\PP}{\mathbb{P}}
\newcommand{\RR}{\mathbb{R}}
\newcommand{\R}{\mathbb{R}}
\newcommand{\QQ}{\mathbb{Q}}
\newcommand{\CC}{\mathbb{C}}

\newcommand{\C}{\mathcal{C}}
\newcommand{\D}{\mathcal{D}}

\title{\bf Maximal Mumford Curves \\ from Planar Graphs}
\author{Mario Kummer, Bernd Sturmfels and Raluca Vlad}
\date{}
\begin{document}
\maketitle

\begin{abstract}
\noindent
A curve of genus $g$ is maximal Mumford (MM) if it has $g+1$ ovals
 and $g$ tropical~cycles.
We construct full-dimensional families of MM curves in 
the Hilbert scheme of canonical curves.
This rests on first-order deformations of graph curves
whose graph is planar.
 \end{abstract} 
 
 \section{Introduction}

In real projective geometry, an algebraic curve $C$ 
consists of ovals, each topologically a circle.
By Harnack's Theorem, the number of
ovals is at most $g+1$, where $g$ is the genus of $C$.
If equality holds, then $C$ is a {\em maximal curve} or {\em M-curve}
\cite{AG, Viro}.
In non-archimedean geometry, a curve $C$ is tropicalized to
a metric graph, known as the
Berkovich skeleton.
The number of cycles in this graph is at most $g$.
If equality holds, then $C$ is a {\em Mumford~curve}~\cite{Bra, Jell, MR}.

This paper concerns curves that attain both bounds simultaneously.
These are defined over a non-archimedean real closed field,
such as the field  $ \RR\{\!\{ \epsilon \} \!\}$ of real Puiseux series.
A smooth algebraic curve over this field is an {\em MM-curve} if it is both
maximal and Mumford.

There is a well-known method, due to Viro \cite{Viro}, for constructing
MM-curves in the plane or, more generally, in toric surfaces. {\em Viro's method} starts with signs
for the lattice points in the Newton polygon and it identifies
regular unimodular triangulations whose associated real tropical curve
has $g+1$ ovals. This has been
highly successful in the study of Hilbert's 16th problem.
Another construction of MM-curves, starting from plabic graphs, was introduced by
Abenda and Grinevich \cite{AG19, AG} in the context of the integrable systems.
One limitation of all planar constructions is that 
plane curves are special. The planar locus
in the moduli space  $\mathcal{M}_g$ is low-dimensional.
Another drawback of Viro curves is that the
coefficients of their equations, when specialized to $\RR$,
 are extremely large or
extremely small. This makes them impractical for
numerical  algebraic geometry.
 Viro's method extends to  complete intersections, as shown in \cite{Stu},
 but the resulting curves suffer from similar shortcomings.

In this article we present a new algebraic construction of MM-curves.
This offers a fresh perspective. Our MM-curves are canonical curves
in $\PP^{g-1}$, obtained by deforming graph curves. This works if and only if the
underlying graph is planar, and it yields full-dimensional loci in $\mathcal{M}_g$.
In  the following examples we display some equations with friendly coefficients.
We used the software {\tt Bertini} \cite{bertini} to
verify the number of ovals for
(\ref{eq:genus3}), (\ref{eq:genus4}), (\ref{eq:genus5}) and
(\ref{eq:limitgen6}).

\begin{example}[$g=3$] \label{ex:g3} Canonical curves of genus three are plane quartics. The equation
\begin{equation}
\label{eq:genus3} xyz(x+y+z) \, = \, \epsilon \cdot f(x,y,z), 
\end{equation}
defines a smooth MM-curve if and only if
$f$ is positive at the six points
$(1{:} 0{:}0)$, $(0{:}1{:}0)$, $(0{:}0{:}1)$,
$(1{:}{-}1{:}0)$, $(1{:}0{:}{-}1)$ and $(0{:}1{:}{-}1)$ in $\PP^2(\RR)$.
Here $f$ is a quartic with real coefficients. 
 For instance, the Fermat quartic $f = x^4 + y^4 + z^4$ yields an MM-curve.
For $\epsilon = 0$, equation (\ref{eq:genus3})
defines the graph curve of the complete graph $K_4$,
which has the six nodes listed above.
\end{example}

\begin{example}[$g=4$]  \label{ex:g4}
The following canonical curve of genus four in $\PP^3$ is an MM-curve:
\begin{equation}
\label{eq:genus4}
 (x\!+\!y\!+\!z\!+\!w)w \,-\, \epsilon \cdot (xy+xz+yz) \,\,\, = \,\,\,  xyz  \,  -\, \epsilon\cdot (x^3+y^3+z^3) \,\,\, = \,\,\, 0. 
\end{equation}
This curve is Mumford because the special fiber for $\epsilon = 0$ consists of six lines.
The tropical curve is the edge graph of the triangular prism. The curve is maximal
because it has five ovals for $0 \!<\! \epsilon\! <\! 0.08860579084$. This is
the smallest positive root of the discriminant.
\end{example}

\begin{example}[$g=5$] \label{ex:g5} The following canonical curve of genus five in $\PP^4$ is an MM-curve:
\begin{equation}
\label{eq:genus5}
x y \, - \, \epsilon \cdot (z^2 + v^2) \,\, = \,\,  zv \, - \, \epsilon \cdot (x^2 + y^2)
\,\,= \,\,  w (x+y+z+v+w) \, - \,\epsilon \cdot (x+y)(z + v) \,\, = \,\, 0 .\qquad
\end{equation}
The curve is maximal: it has $6$ ovals.
It is also Mumford: the tropical curve has $5$ loops.
This curve is not trigonal. By contrast, all planar curves 
land in the trigonal  divisor of~$\mathcal{M}_5$.
\end{example}

We now state the main result of this paper. 
As in \cite{BE} and \cite[Section 5]{GHSV}, we fix
a trivalent $3$-connected simple graph $G$ of genus $g$.
The associated  {\em graph curve} $C_G $
consists of $2g-2$ lines  in $\PP^{g-1}$, one for each vertex of $G$.
It has $3g-3$ nodes, one for each edge of~$G$.
The special fibers seen above
are graph curves for $g=3,4,5$.
The graph $G$ in (\ref{eq:genus5})  is the $3$-cube.

\begin{theorem} \label{thm:main}
An MM-curve in $\PP^{g-1}$ with special fiber $C_G$ exists 
if and only if $G$ is planar. 
The locus of such MM-curves
is full-dimensional in the Hilbert scheme of canonical curves.
Every first order deformation from an open cone, isomorphic to 
$\RR^{g^2-1} \times \RR^{3g-3}_{> 0}$,
in the tangent space of the Hilbert scheme at $C_G$ lifts to an MM-curve.
And, every such lift is an MM-curve.
\end{theorem}

We explain this theorem for  the $3$-cube graph in 
Example \ref{ex:g5}. Here the Hilbert scheme has dimension $36$.
Its quotient by the $24$-dimensional group
${\rm PGL}(5)$ gives rise to the $12$-dimensional moduli space $\mathcal{M}_5$.
We consider the family of canonical curves  which is given~by
\begin{equation}
\label{eq:genus5b}
 xy\,=\,\epsilon \cdot f \, , \,\,\, zv \,=\, \epsilon \cdot g \, , \,\,\,
w(x{+}y{+}z{+}v{+}w) \,=\, \epsilon \cdot h. 
\end{equation}
The coefficients of the three trailing quadrics furnish $36$ parameters for the Hilbert scheme:
$$
 \begin{small} \begin{matrix}
  f & = & a_1 x^2 + a_2 x z+a_3 x v+a_4 y^2+a_5 yz+a_6 yv+a_7 y w+a_8 z^2
  +a_9 z w+a_{10} v^2 + a_{11} v w + a_{12}w^2 ,\\
g &=& b_1 x^2+b_2 x z+b_3 x v+b_4 y^2+b_5 yz+b_6 yv+b_7 y w+b_8 z^2
  +b_9 z w+b_{10} v^2+b_{11}v w+b_{12}w^2 ,\\
h & = & c_1 x^2+c_2 x z+c_3 x v+c_4 y^2+c_5 yz+c_6 yv+c_7 y w+c_8 z^2
  +c_9 z w+c_{10} v^2+c_{11}v w+c_{12}w^2 .
  \end{matrix} \end{small}
  $$
  Our open cone $\RR^{24} \times \RR^{12}_{> 0}$ 
  is defined by the following $12$ linear inequalities in these parameters:
   \begin{equation}
 \label{eq:12ineqs}
  \begin{matrix} \!\!\!\!\!\!\!\!\!
a_8\,,  \,a_8-a_9+a_{12}\,, \,a_{10} \,, \,a_{10}-a_{11}+a_{12}\,, \,\,\,
  b_1\,,\, b_1+b_{12}\,, \,
 b_4\,, \,b_4-b_7+b_{12},\\ 
 \qquad  \quad
 -c_4+c_6 -c_{10}\,,\, \,-c_4+c_5-c_8\,, \,\,-c_1+c_3 -c_{10}\,, \,\,-c_1+c_2-c_8 \,\,\,
 > \,\,\, 0 .
\end{matrix}     
\end{equation}
Since the curve is a complete intersection, we locally
identify the Hilbert scheme at $C_G$ with its tangent space $\RR^{36}$.
Theorem \ref{thm:main} tells us
that (\ref{eq:genus5b}) is an MM-curve whenever (\ref{eq:12ineqs}) holds.

We note that the $12$ linear forms in (\ref{eq:12ineqs}) are relevant not just over $\RR$, but
 over any field $K$.
If they are non-zero then the curve is smooth,
and it is a Mumford curve if ${\rm val}_K(\epsilon) > 0$.

\smallskip

This paper in organized as follows. In Section \ref{sec2} we review non-archimedean real geometry,
starting with the identification in Theorem~\ref{thm:germs}
of algebraic Puiseux series with germs of curves.
We focus on the topology of real curves and their deformations.
Proposition \ref{prop:genusone} characterizes 
the mixed semialgebraic set of MM-curves in genus one.
Section \ref{sec3} is devoted to graph curves and their neighborhood in the 
Hilbert scheme of canonical curves. We prove in Theorem~\ref{thm:main34}
that this neighborhood contains an MM-curve if and only if the graph is planar.
This rests on graph theory results from
the 1930's, due to Whitney and MacLane~\cite{MacLane1937}.

In Section \ref{sec4} we introduce an arrangement of
$3g-3$ hyperplanes in the tangent space of the
Hilbert scheme at a graph curve. Each region of this arrangement
is a cone $\RR^{g^2-1} \times \RR_{>0}^{3g-3}$, and it parametrizes
first-order deformations which lift to smooth Mumford curves (Theorem~\ref{thm:makeMumford}).
The proof of Theorem \ref{thm:main} is completed in Section \ref{sec5}.
Here we construct MM-curves, using
 the combinatorial model of real Mumford curves
via edge pairings given in Proposition~\ref{prop:determined}.
 Putting our theory of MM-curves into practice
 is the aim of Section \ref{sec6}.
 We discuss computations for planar graphs~$G$ 
 whose curves $C_G$ are not complete intersections.

\section{Real Algebraic Geometry with Valuations}
\label{sec2}

In this section we present basics in
non-archimedean real algebraic geometry \cite{AGS, Jell, JSY}.
Many  constructions and results in this paper make no reference to the order of the field.
Such results will be valid over any field with non-trivial valuation, such as
the $p$-adic numbers~$\QQ_p$.

But, in order to introduce MM-curves, we need a valued field $R$ that is also ordered.
  We must assume that the valuation
 ${\rm val}$ is {\em convex} with respect to the
order. This means~that
$$  0 \leq u \leq v \,\,\,{\rm and} \,\,\, v \in \mathcal{O}
\quad {\rm implies} \quad u \in \mathcal{O}. $$
Here $\mathcal{O}$ denotes the valuation ring.
Also, we want
 ${\rm val}$ to be non-trivial in the sense
that there exists a field element $\epsilon > 0$ with ${\rm val}(\epsilon) = 1$.
This ensures that the value group  contains $\QQ$.
The familiar real closed field $ \RR\{\!\{ \epsilon \} \!\}$ of Puiseux series 
satisfies all these desiderata.

We here prefer to work over a subfield, namely 
the field of \emph{algebraic} Puiseux series:
\begin{equation*}
 R\,:=\,\R\langle\epsilon\rangle\,=\,\bigl\{\,f\in\R\{\!\{ \epsilon \} \!\}
 \, \,:\,\, f\textrm{ is algebraic over }\R(\epsilon) \bigr\}.
\end{equation*}
The field $R$ is the real closure of the field of rational functions $\R(\epsilon)$,
 with respect to the ordering where $\epsilon$ is positive but smaller than all positive reals. It has 
the natural $\epsilon$-adic valuation ${\rm val}\colon R^*\to\R$.
The valuation ring is $\mathcal{O} \,:= \, \{ f \in R \,:\,{\rm val}(f) \geq 0\}$, 
with maximal ideal $\mathcal{M} \,:= \, \{ f \in R \,:\,{\rm val}(f) > 0\}$, 
and the residue field $\mathcal{O}/\mathcal{M}$ is the field $\RR$ of real numbers.

Note that $R$ contains the rational function field $\QQ(\epsilon)$.
We can perform computations
in $\QQ(\epsilon)$. The points of
our varieties have their coordinates in~$R$.
It is often desirable to pass to $\RR$ by
replacing  $\epsilon$ with small positive reals.
This is possible for the field $R$, but not for $\R\{\!\{ \epsilon \} \!\}$.
Indeed, our field $R$ has the following key properties.
See \cite[Section 3.3]{basuetal} for proofs.

\begin{theorem} \label{thm:germs}
Let $\,\alpha\colon (0,1)\to\R$ be a continuous semialgebraic function.   \vspace{-0.13cm}
  \begin{enumerate}
   \item  There exists a real number $c \in (0,1)$ and a unique algebraic Puiseux series $f_{\alpha}\in R$ which converges on the interval $(0,c)$ and whose value agrees with $\alpha$ on $(0,c)$.
   \vspace{-0.13cm}
      \item If the limit $\,\lim_{t\to 0} \alpha(t)\,$ exists in $\RR$, then  the series
            $f_{\alpha}$ lies in the valuation ring $\mathcal{O}$.   \vspace{-0.1cm}
   \item Every $f\in R$ equals $f_{\alpha}$ for some continuous semialgebraic function $\alpha\colon (0,1)\to\R$.   
    \vspace{-0.13cm}
   \item If $\alpha_1,\alpha_2\colon (0,1)\to\R$ are continuous semialgebraic functions with $f_{\alpha_1}=f_{\alpha_2}$, then there exists a real number $c \in (0,1)$ such that $\alpha_1(t)=\alpha_2(t)$ for all 
   $t$ in the interval $(0,c)$.
  \end{enumerate}
\end{theorem}

Let $X$ be a quasi-projective scheme over the field $R$. Consider any
point $\nu \in X(\mathcal{O})$. This is a
morphism $\nu\colon{\rm Spec}(\mathcal{O})\to X$, and it
induces morphisms ${\rm Spec}(\R)\to X$ and ${\rm Spec}(R)\to X$.
These morphisms are called the {\em special fiber} and the {\em general fiber} of $\nu\in X(\mathcal{O})$,
respectively.

Consider two continuous semialgebraic maps $\alpha_1,\alpha_2\colon (0,1)\to X(\R)$. 
These are paths in the variety of real points in $X$.
We define an equivalence relation $\,\simeq\,$ on such paths as follows:
\begin{equation*}
 \alpha_1\simeq \alpha_2 \,\,:\,\,\Longleftrightarrow \,\,\exists \,c \in (0,1) \,\,\forall \,t \in (0,c) \,\,\,
 \alpha_1(t)=\alpha_2(t).
\end{equation*}
After covering the $R$-scheme $X$ with affine charts, Theorem \ref{thm:germs} gives a natural bijection
\begin{equation}\label{eq:bijgerms}
 \bigl\{\textrm{Continuous semialgebraic paths } \, [0,1)\to X(\R)\bigr\}/\!\simeq
 \quad \longleftrightarrow \quad X(\mathcal{O})
\end{equation}
such that $\alpha(0)\in X(\R)$ is the special fiber of the point in $X(\mathcal{O})$ associated to a path $\alpha$.

The affine space $R^n$ has two semialgebraic structures.
The first structure comes from  $R$  being an ordered field.
We say that $S \subset R^n$ is {\em order semialgebraic} if it is
defined by  a finite Boolean combination of inequalities $f(u) \geq 0$
where $f $ is a polynomial in $R[x] = R[x_1,\ldots,x_n]$.
The second structure  comes from $R$ being a valued field;
see \cite[Section 2.1]{Nic} for some basics.  We say that $S$ is {\em valuation semialgebraic}
if it is defined by a finite Boolean combination of inequalities
${\rm val}(g(u)) \geq {\rm val}(h(u))$ where $f,g \in R[x]$.
In what follows we are mixing the two notions of semialgebraicity.
We say that $S \subset R^n$ is {\em mixed semialgebraic}
if $S$ is defined by a Boolean combination of inequalities
$f(u) \geq 0$ and ${\rm val}(g(u)) \geq {\rm val}(h(u))$,
where $f,g,h \in R[x]$.

Our study on MM-curves combines the two settings of
semialgebraicity, and identifies curves
that are extreme in both of them. We reiterate the
definitions. A smooth projective curve $C$ 
of genus $g$ over $R$ is an {\em M-curve} if  it
has $g+1$ semialgebraically  connected components.
These components are called ovals and their number
$g+1$ is the maximal possible, thanks to Harnack's Theorem.
Recall that $C$ is a {\em Mumford curve} if there exists a semistable reduction such that the dual graph of its special fiber $C_\RR$ over the residue field $\RR$ has genus~$g$. Note that, for example by \cite[Proposition~3.47]{moduli}, one could equivalently require that, for \emph{every} semistable reduction, the dual graph of the special fiber has genus $g$.
The vertices of this dual graph are the irreducible components of
$C_\RR$, there is a proper edge for every intersection point
between two components,  and there is a loop for every 
node on a component. The genus is the number of
edges minus the number of vertices plus one.
Morrison and Ren \cite{MR} offered an algorithmic approach, and
Jell \cite{Jell} showed that every Mumford curve
admits a faithful tropicalization.
We say that an M-curve $C$ is an {\em MM-curve} if it is also
Mumford.

\begin{proposition}
The set of MM-curves  is mixed semialgebraic
 in the moduli space $\mathcal{M}_g(R)$.
\end{proposition}

\begin{proof}[Sketch of Proof]
The first ingredient is that the locus of M-curves is a semialgebraic subset of
$\mathcal{M}_g(\RR)$, defined by polynomials over $\RR$.
This follows from the results of Sepp\"al\"a and Silhol in \cite{SS}.
Since $R$ is an ordered field extension of $\RR$,
we obtain same result also for $\mathcal{M}_g(R)$.

 The second ingredient is the structure of the moduli space of all Mumford curves,
 over any non-archimedean field.
This space is an open analytic subset of the moduli space.   For a nice
explanation see the recent article of Bradley \cite[Section 2.2]{Bra}.
This goes back to the 1980s, in the work of Gerritzen \cite{Ger} and Herrlich \cite{Her}.
The universal family over this subset was studied by Ichikawa in \cite{Ich}.
The analytic construction
 implies that Mumford curves are characterized by inequalities on the valuations of
moduli coordinates of the curve.
Hence the Mumford curves form a valuation semialgebraic
subset of $\mathcal{M}_g(R)$. 
It would be interesting to find a constructive proof via the
tropical techniques developed by
Gunn and Jell in \cite{GJ}.
\end{proof}

We illustrate this result for genus one curves,
realized as cubic curves in the plane $\PP^2(R)$:
$$ f(x,y,z) \,=\,c_1 x^3 + c_2 x^2 y + c_3 x^2 z + c_4 x y^2 + c_5 x yz + c_6 x z^2 + 
c_7 y^3 + c_8 y^2 z + c_9 y z^2 + c_{10}z^3. $$
The {\em discriminant} $\Delta(f)$ is a homogeneous polynomial
with $2040$ terms of degree $12$ in
 $c_1,\ldots,c_{10}$, including
$ -19683 c_1^4 c_7^4 c_{10}^4$.
  The {\em Aronhold invariant} $A(f)$ is a quartic with $25$ terms in $c_1,\ldots,c_{10}$,
  including $c_5^4$. The {\em j-invariant} of the cubic curve $V(f)$ is the rational function
  $$ j(f) \,\, = \,\,  \frac{A(f)^3}{\Delta(f)}. $$
  
  \begin{proposition} \label{prop:genusone}
  The plane cubic curve $V(f) \subset \PP^2 (R)$ is an MM-curve if and only if
  \begin{equation}
\label{eq:genusone}
  \Delta(f) > 0 \quad {\rm and} \quad {\rm val}\bigl( \,j(f)\, \bigr) < 0, 
  \end{equation}
  i.e.~the discriminant is positive in $R$ and the valuation of the j-invariant 
  is negative in $\QQ$.
  \end{proposition}
  
  \begin{proof}
  We can  transform $f$ into Weierstrass normal form over $R$.
  This scales $\Delta(f)$ by a positive scalar, so it fixes
  the sign of $\Delta(f)$.    In Weierstrass normal form, it is easy to check that
  $V(f)$ has two ovals if and only if $\Delta(f)  > 0$.
  The fact that ${\rm val}(j(f)) $ is negative if and only if 
  $V(f)$ is a Mumford curve is well known in tropical geometry.
  See \cite{CS} for a refinement.
    \end{proof}

Finding a
characterization like (\ref{eq:genusone}) is considerably more
challenging in higher genus, even for hyperelliptic curves.
For instance, consider the genus two curve defined by $y^2 = p(x)$, where $p(x)$ has degree six.
This  is an M-curve if and only if all six roots of $p(x)$ are real.
It is a Mumford curve if and only if condition
III or IV or VII holds in \cite[Theorem 2.11]{Hel}.
We are hopeful that the faithful embedding results in \cite{GJ, Jell} will
lead to progress when $g \geq 3$.

In this article, we assume $g \geq 3 $, and we focus on
curves that are canonically embedded into $\PP^{g-1}$.
To set the stage, we first consider curves of genus $g$ and any degree $d$ in
any $\PP^n$.
Let $H$ be the Hilbert scheme of these curves.
We  shall examine the bijection from (\ref{eq:bijgerms}) for $H$.

 Let $x\in H(\mathcal{O})$ be an $\mathcal{O}$-valued point with general fiber $y\in H(R)$ and $\alpha\colon [0,1)\to H(\R)$ a corresponding continuous semialgebraic path in the Hilbert scheme. Let $Y$ be the curve over $R$ corresponding to $y$. For every $t\in[0,1)$, we write $X_t$ for the curve over $\R$ 
 that corresponds to $\alpha(t)$. We want to relate the topology of
 the non-archimedean curve $Y(R)$ to that of the real curve $X_t(\R)$ for small $t>0$. For this purpose we will employ semialgebraic triviality.

\begin{theorem}\label{thm:trivial}
In the situation above,
    there exists a real constant $c \in (0,1)$,  an inclusion  $F'\subset F$ 
    of semialgebraic sets, and a continuous semialgebraic map
    \begin{equation*}
        \psi \,\colon \, F\times(0,c)\,\,\to \,\,\PP^n(\CC)
    \end{equation*}
    with the following property:
    for each real constant $t \in (0,c)$, the map $\psi_t\colon F\to\PP^n(\CC)$ is a semialgebraic homeomorphism onto the complex curve $X_t(\CC)$ which maps $F'$ onto $X_t(\R)$.
\end{theorem}

\begin{proof}
    There exists $c' \in (0,1)$ and a semialgebraic subset $S\subset\PP^n(\CC)\times(0,c')$ such that the fiber over every $t\in(0,c')$ under the projection onto the second coordinate is $X_t(\CC)$. We
    now apply \cite[Theorem~9.3.2]{rag} to obtain 
    $c\in(0,c')$, $F'\subset F$ and $\psi$ having the stated properties.
        \end{proof}

The next corollary says that relevant topological features over $\RR$ agree
with those over~$R$. This uses the notion of connectivity in
\cite[Section 5.2]{basuetal} and not the topologies in
\cite[Section~2]{JSY}.

\begin{corollary}
    The number of semialgebraically connected components of $Y(R)$ agrees with the number of connected components of  the real curve $X_t(\R)$ for any value $t$ in $ (0,c)$.
\end{corollary}

\begin{proof}
    The continuous semialgebraic map $F'\times(0,c)\to\PP^n(\R)$ from Theorem \ref{thm:trivial}
     lifts to a map $\varphi\colon F'_R\times(0,c)\to\PP^n(R)$ over $R$. Here $F'_R$ is a lift of $F'$ in the sense that it is described over $R$ by the same inequalities as $F'$ over $\R$. Then $\varphi_\epsilon\colon F'_R\to Y(R)$ is a semialgebraic homeomorphism. The number of semialgebraically connected components of $F'_R$ and $F'$ 
     coincide. The latter equals the number of connected components of $X_t(\R)$ for $t \in (0,c)$.
\end{proof}

We now fix a map $\psi$ as in Theorem \ref{thm:trivial}, and we study what happens when 
$t$ tends to~$0$. We know from \cite[Proposition~II-29]{geometryofschemes}  that $X_0(\CC)$ is the Hausdorff limit of $X_t(\CC)$. This shows that $\lim_{t\to 0}\psi(a,t)$ exists in $X_0(\CC)$ 
for every $a\in F$.
Conversely, every point in $X_0(\CC)$ can be written as such a limit. Thus, we have a surjective continuous semialgebraic map
\begin{equation*}
    \widetilde{\psi}\colon F\to X_0(\CC)\,,\,\, a\mapsto \lim_{t\to 0}\psi(a,t).
\end{equation*}
It satisfies $\widetilde{\psi}(F')\subset X_0(\R)$, but equality can fail in general.
However, 
 $\widetilde{\psi}(F')= X_0(\R)$  does hold when the smooth points of $X_0(\R)$ are dense (in the Euclidean topology) in $X_0(\R)$.
To see this, we note that
  the fiber of $\widetilde{\psi}$ over any smooth real point has only one element.
  That point must be real because
    nonreal preimages of real points come in complex conjugate pairs.
  
We close this section by reviewing how tangent spaces can be described using dual numbers.
This will later be applied to Hilbert schemes. In what follows, $H$ can be any quasi-projective scheme over $\R$.
We write $D=\RR[\epsilon]/(\epsilon^2)$ for the {\em ring of dual numbers} over $\RR$.
 There is a natural bijection between $H(D)$ and pairs $(x,\eta)$ where $x\in H(\RR)$ is a real point and $\eta\in T_x(H)$ is a tangent vector 
 to $H$ at $x$.
 This is the content of \cite[Exercise~II.2.8]{hartshorne}.
 
 Consider the  ring of algebraic power series over the real numbers $\R$. This is the subring
 \begin{equation}
 \label{eq:algseries}
 A \,\,:=\,\, \{f\in\R[[\epsilon]]\,:\, f\textrm{ is algebraic over }\R(\epsilon)\} \quad \subset \quad \mathcal{O}\,\,
 \subset \,\, R.
\end{equation}
There is a natural ring isomorphism $A/(\epsilon^2)\to D$. This induces a map $H(A)\to H(D)$. Therefore, every point in $H(A)$ gives rise to a pair $(x,\eta)$ where $x\in H(\R)$ and $\eta\in T_x(H)$. 
If $x$ is a smooth point of the scheme $H$ then every tangent vector arises in this way:

\begin{lemma}\label{lem:smoothtangent}
 Let $x\in H(\R)$ be a smooth point on $H$ and $\eta\in T_x(H)$ any tangent vector. There exists a point in 
 $H(A)$ whose image under the natural map $H(A)\to H(D)$ is $(x,\eta)$.
\end{lemma}

\begin{proof}
 We may assume that $\eta \neq 0$.
 Since $x$ is a smooth point, there is a smooth curve $C\subset H$ through $x$ and a tangent vector $0\neq\eta'\in T_xC$ which maps to $\eta$ under the map $T_xC\to T_xH$ induced by inclusion. Hence it suffices to prove the claim for the case when $H$ is a curve.
 In that case,
   Lemma \ref{lem:smoothtangent}
  follows from the fact that the completion of the local ring of a curve at a smooth 
  $\RR$-valued point is isomorphic to the power series ring $\R[[\epsilon]]$.
\end{proof}

\section{Graph Curves and Planar Graphs}
\label{sec3}

Let $G$ be a $3$-connected trivalent simple graph of genus $g$.
Such a graph has $2g-2$ vertices and $3g-3$ edges.
The associated graph curve $C_G$ consists of $2g-2$ lines
in $\PP^{g-1}$. Two lines intersect precisely when the corresponding
vertices in $G$ are connected by an edge.
Graph curves were introduced by Bayer and Eisenbud \cite{BE}, and they
 are unique up to projective transformations of $\PP^{g-1}$.
Their coordinates are given in \cite[Proposition 5.2]{GHSV}.
The MathRepo page \url{https://mathrepo.mis.mpg.de/selfdual/}
offers {\tt Macaulay2} code for creating the ideals of $C_G$ for graphs $G$ of genus $g \leq 7$.
For instance, when $G$ is the edge graph of the $3$-cube
then the ideal is
$\langle xy, zv, w(x{+}y{+}z{+}v{+}w) \rangle$;
see Example \ref{ex:g5} and \cite[Example 5.1]{GHSV}.
Further examples of graphs $G$ are shown in Figure  \ref{fig:G81b}
for $g=5$ and in Figure \ref{fig:assoc} for $g=8$.

Let $H_g$ denote the Hilbert scheme of curves
of genus $g$ and degree $2g-2$ in $\PP^{g-1}$.
Note that the curve $C_G$ is a $\QQ$-valued point in $H_g$. We are interested in the following question.

\begin{question}\label{qu:graphqu}
 For which graphs $G$, does there exist a point $x $ in the Hilbert scheme $H_g(\mathcal{O})$ 
whose curve $Y$  over  $R$ is an MM-curve 
 and  whose special fiber $X_0$ is the graph curve $ C_G\,$?
 \end{question}

We saw three such graphs $G$ in the Introduction.
The desired point  $x$ in $H_g(\mathcal{O})$ corresponds to
a continuous semialgebraic path $\alpha\colon[0,1)\to H_g(\R)$ such that $\alpha(0) = X_0$ equals
$C_G$ and $\alpha(t)$ is a
smooth maximal curve $X_t$ over $\RR$ for sufficiently small $t>0$. 
Given such a path $\alpha$, we now explain how to
encode the topology of $X_t$ in terms of combinatorial data. By Theorem \ref{thm:trivial}, there is a semialgebraic set $S$ and a continuous semialgebraic surjection
\begin{equation}\label{eq:varphi}
    \varphi\,\colon \,S\,\to\, X_0(\R)
\end{equation}
such that $S$ is homeomorphic to $X_t(\R)$ for  small $t>0$ and whose fiber over a smooth point has cardinality one. Since $X_t$ is smooth, the set $S$ is a disjoint union of $r$ circles $S_1,\ldots,S_r$.

We write $e_1,\ldots,e_{3g-3}$ for the edges of $G$. The real part of the graph curve $X_0 = C_G$ is
an arrangement of $2g-2$ circles in $\PP^{g-1}(\RR)$. Each circle intersects three other circles.
Each circle is therefore the union of three line segments, each bounded by two intersection points.
Let $L_{ij}$ denote the open line segment bounded by the nodes corresponding to the edges $e_i$ and $e_j$
of the graph $G$. Hence the circle corresponding to vertex $v$ of the graph $G$ is the union of the open line segments $L_{ij}$, $L_{ik}$ and $L_{jk}$, where $e_i$, $e_j$ and $e_{k}$ are the edges adjacent to $v$, and their corresponding nodes.
The closure $L_{ij}'$ of  each inverse image $\varphi^{-1}(L_{ij})$ is a curvy closed line segment.
 The real algebraic curve $S$ is the union of these curve segments $L_{ij}'$.
 
 Fix an edge $e_i$ of $G$ with vertices $v$ and $v'$.
 Let $e_{j_1}$ and $e_{j_2}$ resp. $e_{k_1}$ and $e_{k_2}$ be the other edges adjacent to $v$ and $v'$.
The node of $X_0$ corresponding to $e_i$ is contained in the closures~of
\begin{equation*}
    L_{ij_1}, \,L_{ij_2}, \,L_{ik_1} \,\, {\rm and} \,\, L_{ik_2}.
\end{equation*}
The closure of the union of the preimages of these line segments under  $\varphi$  equals
\begin{equation}
\label{eq:preimages}
    L_{ij_1}'\cup L_{ij_2}'\cup L_{ik_1}'\cup L_{ik_2}'.
\end{equation}
This semialgebraic curve has precisely two connected components. These are either $L_{ij_1}'\cup L_{ik_1}'$ and $L_{ij_2}'\cup L_{ik_2}'$ or $L_{ij_1}'\cup L_{ik_2}'$ and $L_{ij_2}'\cup L_{ik_1}'$.
Hence the map $\varphi$ induces an edge pairing of $G$.
We define an {\em  edge pairing} to be a map $\rho$ that assigns to each edge $e=(v,v')$ of $G$ a partition
    \begin{equation}
    \label{eq:edgepairing}
        \rho(e)\,\,=\,\,\bigl\{\{e_i,e_j\},\{e_k,e_l\}\bigr\}. 
           \end{equation}
    Here $e,e_i,e_k$ are the three edges adjacent to $v$ and $e,e_j,e_l$ are the three edges adjacent to $v'$.

\begin{example}[$g=5$] \label{ex:twelve}
Fix the graph $G$ with $12$ edges shown in Figure~\ref{fig:G81b}.
It has $2^{12} = 4096$ edge pairings $\rho$. To specify $\rho$, one
makes a choice for each edge $e$. For instance, for 
$e = 47$,
the two choices are $\rho(e) = \{\{34,67\},\{45,78\}\}$ or
$\rho(e) = \{\{34,78\},\{45,67\}\}$. 
One distinguished edge pairing $\rho$ arises from the planar embedding of this graph.
This is shown on the right in Figure~\ref{fig:G81b}.
Here we have $\rho(e) = \{\{34,78\},\{45,67\}\}$, as shown in
red and green.
\end{example}

\begin{figure}[h]
\vspace{-0.15in}
$$ \!\!\!\!\! \!\!
\includegraphics[width = 8.6cm]{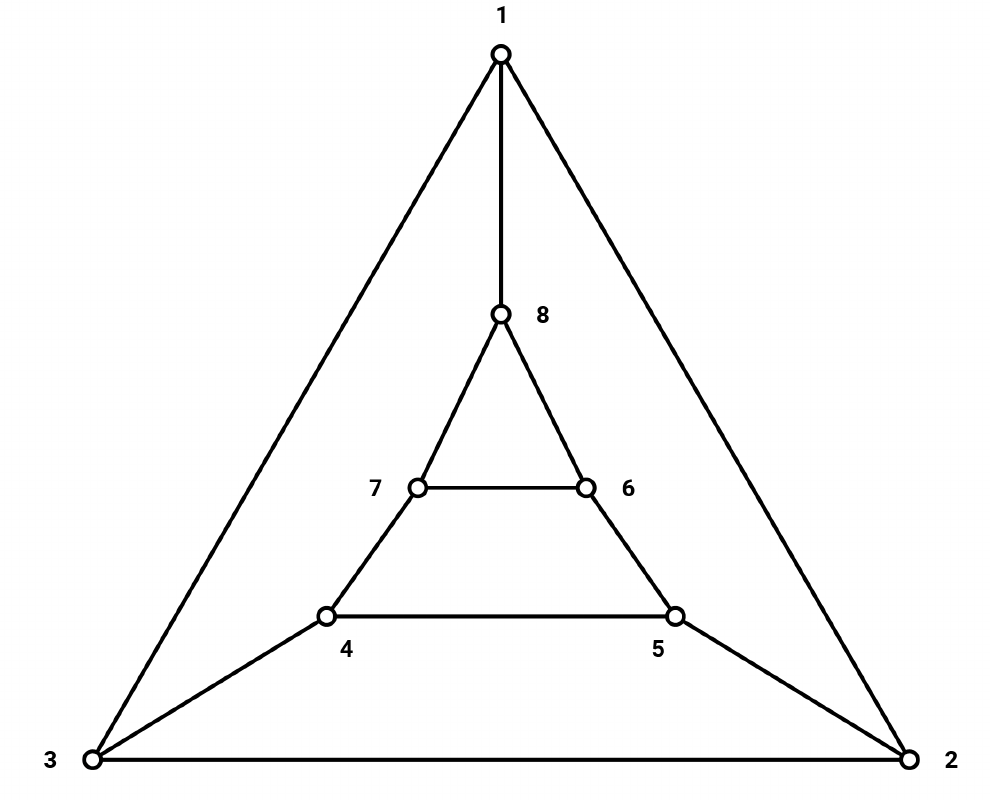} \,\,
\includegraphics[width = 8.9cm]{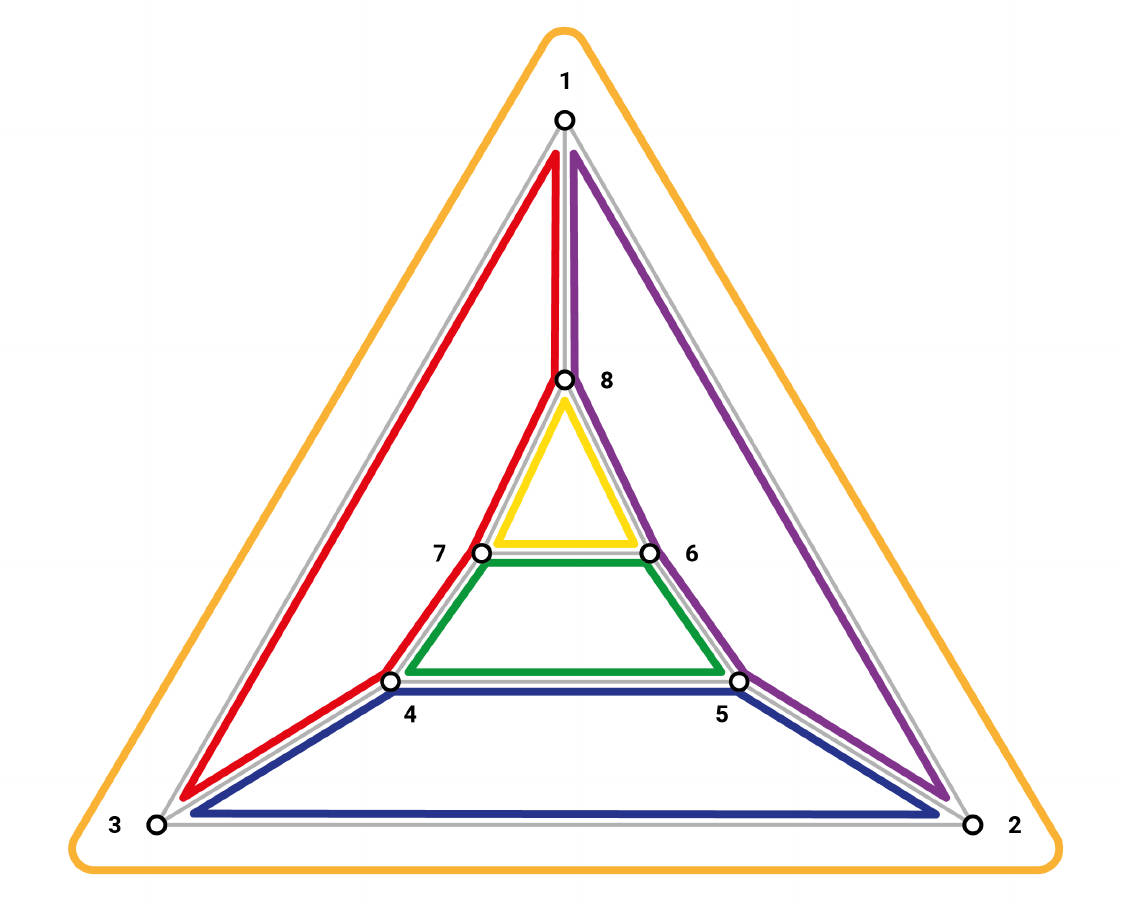} \vspace{-0.15in}
$$
\caption{A genus five graph $G$ and its double-cover $G_\rho \rightarrow G$ with six cycles.}
\label{fig:G81b}
\end{figure}

Given any edge pairing $\rho$ of $G$, we now define a new graph $G_\rho$ as follows.
The vertices of $G_\rho$ are the pairs 
 $\{e_i,e_j\}$ of edges of 
 the given graph $G$ that share a common vertex. Two pairs $\{e_i,e_j\}$ and $\{e_k,e_l\}$ 
 are connected by an edge in the new graph $G_\rho$ if and only if $\{e_i,e_j\}\cap \{e_k,e_l\}=\{e\}$ for some edge $e$ of $G$ where $\,(\{e_i,e_j\}\cup \{e_k,e_l\})\backslash \{e\}\in\rho(e)$. 
An example is shown on the right in Figure \ref{fig:G81b}:
the graph $G_\rho$ consists of the $24$ colorful edges.
The red edge covering $e = 47$ connects the vertices $\{34, e\}$ and $\{78, e\}$ in $G_{\rho}$.
  
  The graph $G_\rho$ is  always a disjoint union of cycles. There exists a natural surjective map
\begin{equation}
\label{eq:2to1map}
 G_\rho \,\to \,G
\end{equation}
that sends the edge connecting $\{e_i,e_j\}$ and $\{e_k,e_l\}$ in $G_\rho$ to
the edge $e$ of $G$ that satisfies
  $\{e\} = \{e_i,e_j\}\cap\{e_k,e_l\}$. Every edge of $G$ is covered by exactly two edges of $G_\rho$.
The vertices of $G_\rho$ correspond to the curvy line segments $L_{ij}'$.
This correspondence respects the topology:

\begin{proposition} \label{prop:determined}
The number of ovals of the smooth curve $S$ in $\PP^{g-1}(\RR)$ 
 is determined by the edge pairing $\rho$ arising from $\varphi$ in 
(\ref{eq:varphi}). To be precise, 
 $S$ is homeomorphic to the graph~$G_\rho$.
\end{proposition}

\begin{proof}
This is a consequence of the decomposition of $S$ near $C_G$ that was described above.
\end{proof}

\begin{corollary}
The curve $S$ is maximal if and only if the graph $G_\rho$ is a union of $g+1$~cycles.
This combinatorial condition is necessary and sufficient for $Y$ to be an MM-curve over $R$.
\end{corollary}

\begin{example}[$g=5$]
Figure \ref{fig:G81d} shows a maximal curve (right) and its graph curve (left)
derived from Example \ref{ex:twelve}. The graph curve $C_G$ consists of eight
circles in $\PP^4(\RR)$. The smooth curve $S$ consists of six pairwise
disjoint ovals. Their colors in Figures \ref{fig:G81b} and \ref{fig:G81d} match.
\end{example}

We now state the main result in this section.
 Theorem \ref{thm:main34}
 is a purely graph-theoretic result, but it
has significant implications for the
topology of real canonical curves in~$\PP^{g-1}$.

\begin{theorem} \label{thm:main34}
A  graph $G$ of genus $g$ as above admits
a double cover $\,G_\rho \rightarrow G$ by $g+1$ cycles
if and only if $G$ is planar. In this case, there is a 
unique edge pairing $\rho$ with this property.
\end{theorem}

\begin{proof}
A cycle of $G$ is a sequence of edges that starts and ends at the same vertex, i.e.
$$v_1 \to v_2 \to \cdots \to v_m \to v_{m+1} = v_1, $$
where $\{v_i,v_{i+1}\} $ are edges of $G$.
A cycle is \emph{simple} if $m \geq 3$ and the vertices $v_1,v_2,\ldots,v_m$ are distinct.
 A {\em double cover} of $G$ is a set $\C$ of cycles
such that each edge of $G$ appears in precisely  two cycles in $\C$.
We first prove the following claim: {\em
For any double cover $\C$ of the genus $g$ graph $G$, we have $|\C| \leq g + 1$. 
If $\,|\C| = g + 1$, then all the cycles in $\C$ are simple.}

	By splitting the cycles in $\C$ that are not simple, we construct a set $\C'$ of simple cycles of $G$ such that each edge appears precisely twice in these cycles. Then 
	 $|\C'| \geq |\C|$.
		For every vertex $v $ of $G$, let $\C'_v $ be the triple of cycles in $\C'$ that contain $v$. 
	Let $\D \subset \C'$ be a non-empty subset of cycles that add up to zero modulo $2$ in the cycle space.
	For every $v \in V$, either $\C'_v \subset \D$ or $\C'_v \cap \D = \varnothing$ holds.
 Therefore, if a vertex $v$ appears in $\D$, so do all   its neighbors and all the cycles $\C'_v$ containing $v$. Since $G$ is connected, this implies that $\D = \C'$.
 We conclude that
every proper subset of $\C'$ is linearly independent in the cycle space modulo $2$.
 Since the cycle space of $G$ is $g$-dimensional, we have
 $|\C'| \leq g+1$.  This implies the inequality $|\C| \leq |\C'| \leq g + 1$.
  Moreover, if equality holds, then $\C = \C'$. So, the claim is proved.

We now prove the first assertion in Theorem \ref{thm:main34}.
Suppose that
$G$ is planar. Fix an embedding of $G$ into the $2$-sphere $\mathbb{S}^2$.
	We obtain a double cover of $G$ by considering the cycles that bound 
	the regions of  $\mathbb{S}^2 \backslash G$. 
	For the converse, let $\C$  be any double cover of $G$ by $g+1$ cycles.
By the claim in the previous paragraph, these cycles are simple and $g$ of them form 
a basis for the cycle space of $G$.
	The planarity of $G$ then follows from Mac Lane's Theorem \cite[Theorem I]{MacLane1937}.
	This theorem states that a graph $G$ is planar if and only if its cycle space has a basis consisting of simple cycles in which every edge appears at most twice.
	
For the second assertion in Theorem \ref{thm:main34}, we fix any
double cover $\C$ of $G$ by $g+1$ cycles. We claim that $G$ admits an embedding into
the sphere $\mathbb{S}^2$ such that the cycles in $\C$ bound the regions of $\mathbb{S}^2\backslash G$.
 For each cycle $C \in \C$, consider a closed disk whose boundary is $C$. 
 Next we glue these closed disks along common edges. The result is (homeomorphic to) a closed $2$-dimensional manifold. Indeed, smoothness along edges holds because each edge appears in precisely two cycles. Smoothness at vertices holds because 
 the graph $G$ is assumed to be trivalent. The Euler characteristic of our surface~is
	$$\chi \,=\, V - E + F \,= \,(2g-2) - (3g-3) + (g+1) \,= \,2. $$
	
There are no non-orientable closed surfaces with $\chi = 2$. Hence our
surface is the sphere $\mathbb{S}^2$. We have constructed an
embedding of $G$ into $\mathbb{S}^2$ where the faces are bounded by our cycles.
In particular, $G$ is a planar graph.
		Finally, the uniqueness of the set $\mathcal{C}$ of cycles is guaranteed by \cite[Proposition 4.2.7]{Diestel}, which states that the cycles bounding the faces of a  
$3$-connected planar graph are precisely those cycles  which do not disconnect the graph.
\end{proof}

\begin{remark}  Since our graph $G$ is always trivalent, such a double cover of $G$
is equivalent to the structure of a 
{\em ribbon graph} on $G$.
Ribbon graphs are important in the theory of
Riemann surfaces; see for example~\cite[Section 1]{MP}.
There the emphasis is on geometry and analysis.
By contrast,  the present paper focuses on algebra
and symbolic computation.
Ribbon graphs arise from embedding graphs
into surfaces, a topic to be discussed in
Remark~\ref{rem:nottheonly}.
\end{remark}

\section{Hyperplane Arrangements at the Hilbert Scheme}
\label{sec4}

Let $H_g$ denote the Hilbert scheme of curves of genus $g$ and degree $2g-2$ in $\PP^{g-1}$.
We have $H_3 = \PP^{14}$, but $H_g $ has many irreducible components for $g \geq 4$.
A description for $g=4$ is found in \cite{Rez}. We are interested in the distinguished
component of $H_g$  which contains all 
smooth canonical curves. This component is known
to have dimension $\, g^2 + 3g-4$; see e.g.~\cite[Corollary 1.45]{moduli}. This number is the sum
of $g^2 -1 = {\rm dim}({\rm PGL}(g))$  and $3g-3 = {\rm dim}(\mathcal{M}_g)$.
These two contributions will be made manifest in our discussion of  real tangent spaces,
and in particular by the open cone in Theorem~\ref{thm:main}.
We begin with the following key lemma.

\begin{lemma} \label{lem:gabi}
Each graph curve $C_G$ defines a smooth point $x_0$ of the Hilbert scheme $H_g$.
\end{lemma}

\begin{proof}
 By the Jacobi criterion for smoothness, it suffices to show smoothness 
 after base change to  $\CC$.
The proof rests on standard arguments in deformation theory, starting from
the ideal sheaf $\mathcal{I}_C$ and structure sheaf $\mathcal{O}_C$
of our graph curve $C= C_G$. The {\em normal sheaf}
$\mathcal{N}_C = {\rm Hom}_{\mathcal{O}_C}\bigl( \mathcal{I}_C / \mathcal{I}_C^2,\mathcal{O}_C \bigr) $
is locally free. We consider the subsheaf $\,\mathcal{N}_C'\,$ 
of $\,\mathcal{N}_C\,$ that is
 the image of the natural map from the restriction of the tangent sheaf $ \mathcal{T}_{\mathbb{P}^{g-1}}$ of $\mathbb{P}^{g-1}$. This map~is
$$ \mathcal{T}_{\mathbb{P}^{g-1}}\otimes\mathcal{O}_C \,
\longrightarrow \, \mathcal{N}_C. $$
The cokernel of this map is a skyscraper sheaf, known as the {\em Lichtenbaum--Schlessinger sheaf}:
$$ 0 \, \longrightarrow \, \mathcal{N}_C' \,
 \longrightarrow \, \mathcal{N}_C \,  \longrightarrow \,  {\textrm{T}}^1_{{\rm Sing}(C)} \,\simeq \, \CC^{3g-3}
 \,\longrightarrow \,0.
 $$
 As shown in the proof of \cite[Proposition 1.1]{HH}, it suffices to prove that  $H^1(\mathcal{N}_C)=0$ and the map $H^0( \mathcal{N}_C) \rightarrow H^0({\mathcal{T}}^1_{{\rm Sing}(C)} )$  is surjective.
 The latter is true since every node of $C$ can be deformed separately, as in the proof of Lemma  \ref{lem:arrangement} below.
 In order for $H^1(\mathcal{N}_C)=0$, it suffices to show 
 $H^1(\mathcal{N}'_C) = 0$.  This vanishing statement can be derived as in
 \cite[Corollary 1.2]{HH}. For this we
use the isomorphism  $H^1(\mathcal{O}_C) \simeq H^0(\mathcal{O}_C(1)^\vee)$,
obtained from the Euler sequence on $\PP^{g-1}$ restricted to $C$,
and the fact  that $C$ is canonically embedded \cite[Corollary 2.2 (ii)]{BE}.
\end{proof}

We now fix a graph curve $C_G$ with corresponding point $x_0 \in H_g$.
Set $S = \RR[x_1,\ldots,x_g]$ 
and let $I_0 \subset S$ be the homogeneous radical ideal of $C_G$. 
The tangent space $T_{x_0}(H_g)$ can be computed 
(e.g.~in {\tt Macaulay2}) as
the degree zero part of the graded $S$-module
${\rm Hom}_S(I_0,S/I_0)$.
We  know from Lemma \ref{lem:gabi} that $ T_{x_0}(H_g)$ is a real vector space of dimension $g^2+3g-4$.

Recall that $A$ is the ring of algebraic power series 
and $D$ is the ring of dual numbers. Let $\eta \in T_{x_0}(H_g)$ be a tangent vector.
By Lemma~\ref{lem:gabi}, $x_0$ is a smooth point.
By Lemma~\ref{lem:smoothtangent}, there exists a point $a\in H_g(A)$ whose image under the 
map $H_g(A)\to H_g(D)$ equals $(x_0,\eta)$. The natural homomorphism $A\to\mathcal{O}$  induces a map $H_g(A)\to H_g(\mathcal{O})$. This maps $a$ to a point $x\in H(\mathcal{O})$ whose special fiber is $x_0$. Let $y$ be its general fiber and $Y\subset \PP^{g-1}$ the corresponding curve over $R$.
This is a curve with special fiber $C_G$. 
 By Proposition  \ref{prop:determined}, if $Y$ is smooth over $R$ then its
      topology is determined by the induced edge pairing on $G$.
      In the next section we show
that the tangent vector~$\eta$ determines the edge pairing.
Here we record:

\begin{remark}
\label{rmk:curveY}
If $\,Y$ is smooth,  then it is a Mumford curve whose genus $g$ graph equals
$G$.
\end{remark}

We shall derive a condition on $\eta$ which guarantees that $Y$ is smooth. To this end, let $I_\epsilon \subset A[x_1,\ldots,x_g]$ denote the homogeneous vanishing ideal of the curve over $A$ corresponding to the point $a\in H_g(A)$. As before,
  $I_0\subset \R[x_1,\ldots,x_g]$ is the ideal of our graph curve $X_0$. We note that $I_0$ is the image of $I_\epsilon$ under the map $A[x_1,\ldots,x_g]\to\R[x_1,\ldots,x_g]$ 
  that is induced by the 
  quotient map $A \to A/(\epsilon) = \R$. If we write an element of $I_\epsilon$ in the form
\begin{equation*}
 f_0+\epsilon\cdot f_1+\epsilon^2\cdot h
\end{equation*}
for some $f_0,f_1\in\R[x_1,\ldots,x_g]$ and $h\in A[x_1,\ldots,x_g]$, then $f_0\in I_0$ and 
$f_1 \equiv
\eta(f_0)$ modulo $I_0$.

\begin{lemma} \label{lem:welldefined}
 Let $e$ be an edge of $G$ and $p_e\in X_0$ the corresponding node of $X_0 = C_G$. Let $L_1, L_2\subset X_0$ be the lines that intersect at $p_e$ and $W_e$   the plane spanned by
 $L_1$ and $L_2$ in $\PP^{g-1}$.
  Fix two linear forms  $l_1$ and $l_2$ in $\RR[x_1,\ldots,x_g]$ such that $L_i = V(l_i) \cap W_e$ for  $i=1,2$.
 \begin{enumerate}
  \item[(1)] There exists a polynomial $f \in \R[x_1,\dots, x_g]$ such that $f\cdot l_1\cdot l_2\,\in\, I_0$ and $f(p_e)\neq0$.
  \item[(2)] Let $\sigma\colon I_0\to\R[x_1,\ldots,x_g]/I_0$ be a homomorphism of graded modules of degree zero. If $\sigma(f\cdot l_1\cdot l_2)(p_e)=0$, then $\sigma(f'\cdot l'_1\cdot l'_2)(p_e)=0$ for every other choice of $f',l_1',l_2'$ as above.
 \end{enumerate}
\end{lemma}

\begin{proof}
 Part (1) holds because $X_0$ and $L_1\cup L_2$ agree on a 
  Zariski open neighborhood of $p_e$. For part (2) let $q_1,\ldots,q_{g-3}$ be linear forms that cut 
  out the plane $W_e$. There exists a polynomial $h$ with $h(p_e)\neq0$ such that $h\cdot q_j\in I_0$ for all $j=1,\ldots,g-3$.
 We can write $\,   l_i\,=\,l_i'+Q_i\,$
 where each $Q_i$ is a linear combination of the linear forms $q_j$. The homomorphism $\sigma$ satisfies
 \begin{eqnarray*}
    &\sigma \bigl( f l_1 l_2 \bigr)(p_e  )\,=\,0\\
   \Rightarrow & \sigma\bigl( \, f (l_1'+Q_1) (l_2'+Q_2) \,\bigr)(p_e)\,=\,0\\
   \Rightarrow & \sigma\bigl(\,f h \,(l_1' l_2'+Q_1l_2'+ l_1' Q_2+Q_1Q_2)\,\bigr)(p_e)\,=\,0\\
   \Rightarrow & \bigl(\sigma( \,f h l_1' l_2')\,+\,\sigma(f h Q_1) \cdot l_2' \,+\,\sigma(f h  Q_2) \cdot l_1' \,+
   \,\sigma(f h Q_2)\cdot Q_1\, \bigr)(p_e)\,=\,0 \\
   \Rightarrow & \sigma\bigl( \,f h l_1' l_2' \,\bigr)(p_e)\,=\,0\\
   \Rightarrow & \bigl(f h\bigr)(p_e) \cdot \sigma\bigl( f' l_1' l_2' \bigr)(p_e)\,=\,0\\
   \Rightarrow & \sigma\bigl( f' l_1' l_2' \bigr)(p_e)\,=\,0.
 \end{eqnarray*}
This establishes the statement (2), and it completes the proof 
of Lemma \ref{lem:welldefined}.
\end{proof}

Lemma \ref{lem:welldefined} allows us to define a hyperplane arrangement 
in the tangent space $T_{x_0}(H_g)$. It consists of
$3g-3$ hyperplanes $h_e$, one for for every edge $e$ of $G$.
 We define $h_e$ to be the set of all elements $\sigma $ in
$T_{x_0}(H_g)$ such that $\sigma\bigl(f\cdot l_1\cdot l_2\bigr)(p_e)\,=\,0$.
The notation is as in Lemma~\ref{lem:welldefined}.

\begin{lemma} \label{lem:arrangement}
 For every edge $e$ of $G$, the set  $h_e$ is a hyperplane through the origin in $T_{x_0}(H_g)$.
  Moreover, the 
 $3g-3$ hyperplanes $h_e$ are in general position in $ T_{x_0}(H_g) \simeq
 \RR^{g^2-1} \oplus \RR^{3g-3}$.
\end{lemma}

\begin{proof}
By definition, $h_e$ is a linear subspace of codimension $\leq 1$ in $T_{x_0}(H_g)$.
We claim the following: {\em
  for every edge $e$ of $G$, there exists a direction $\eta_e\in T_{x_0}(H_g)$ such that $\eta_e\in\cap_{e'\neq e} h_{e'}$ but $\eta_e\not\in h_e$.}
 This claim implies both statements in the lemma. We shall now prove it.
 
 Fix an edge $e = \{v_1,v_2\}$ of $G$.
   As before, let $L_1, L_2$ be the two lines in $X_0$ that intersect at 
   $p_e$ and let $W_e$ be the plane they span in $\PP^{g-1}$.
     Fix linear forms $l_1,l_2\in\R[x_1,\ldots,x_g]$ such that $V(l_i) \cap W_e = 
      L_i$ for $i=1,2$. 
      We denote by $e_i$ and $e_i'$ the other edges adjacent to $v_i$ for $i=1,2$ and by $p_{e_1},p_{e_1'},p_{e_2},p_{e_2'}$ the corresponding nodes of $X_0$. Furthermore, 
      let $l_1',l_2'$ be linear forms that cut out the lines $L_1'$ and $L_2'$ passing through $p_{e_1},p_{e_2}$ and $p_{e_1'},p_{e_2'}$ respectively.
      
       For any $t\in\R$, we consider the quadratic form
 \begin{equation}
 \label{eq:Xt1}
  F_t\,\,:=\,\,l_1\cdot l_2\,+\,t\cdot l_1'\cdot l_2' \quad \in \,\,\, \RR[x_1,\ldots,x_g].
 \end{equation}
 It defines a quadratic curve $Z_t$ in the plane  $W_e$. 
 This conic is smooth for small nonzero $t$.
 
 This allows us to
   smooth the node $p_e$ of the graph curve $X_0$.
For small nonzero $t$, we set
 \begin{equation}
 \label{eq:Xt2}
  X_t\, \,\,:= \,\, (X_0\backslash(L_1\cup L_2))\cup Z_t.
 \end{equation}
 This is a curve of  degree $2g-2$  in $\PP^{g-1}$.
 It is the union of
 $2g-3$  smooth rational curves whose dual graph is the contraction $G/e$ of $G$ along $e$.
 The graph $G/e$ still has genus~$g$.  Hence the curve
  $X_t$ has arithmetic genus $g$. In particular,
  $X_t$ is in the Hilbert scheme $H_g$.
 
 It follows from \cite[Proposition~III.56]{geometryofschemes} that $X_t$ is a flat family of
 curves over the  $t$-line.
   Its special fiber at $t=0$ is $X_0$.
 This means that the family $X_t$
  induces a direction $\eta_e\in T_{x_0}(H_g)$. 
  This tangent direction satisfies the properties in the claim.
Namely,
  we have $\eta_e \not\in h_e$, because the quadric
    $l_1'\cdot l_2'$ does not vanish at $p_e$. Moreover,
    we have $\eta \in h_{e'}$ for all other edges $e'$ of $G$, because
    $X_t$ agrees with $X_0$ locally around the node $p_{e'}$.
    This completes the proof.
\end{proof}

We now present our main result in this section.
It refers to the curve $Y$ in Remark~\ref{rmk:curveY}.

\begin{theorem} \label{thm:makeMumford}
Let $\eta \in T_{x_0}(H_g)$ be a tangent vector that does not lie on any of
the $3g-3$ hyperplanes $\,h_e$. Then $Y$
is a smooth Mumford curve whose special fiber is the graph curve~$C_G$.
\end{theorem}

\begin{proof}
 Let $C=R(\sqrt{-1})$ be the algebraic closure of the real closed field $R$. The reduction~map 
$\,  \PP^{g-1}(C)\to\PP^{g-1}(\CC)\,$
 restricts to a reduction map $ Y(C)\to X_0(\CC)$. Suppose that $p\in Y(C)$ is a singular point. 
 Our goal is to derive a contradiction from this assumption.
 
 The reduction of $p$ is a singular point $p_e$ of $X_0(\CC)$,
  corresponding to an edge $e$ of $G$.
 After changing coordinates in $\PP^{g-1}$,
  the plane $W_e$ spanned by the lines $L_1 $ and $L_2$ intersecting at $p_e$ is the zero set of $x_3,\ldots,x_{g-1}$ and $L_i$ is defined in $W_e$ by $x_i=0$ for $i=1,2$. This implies
 $$ \qquad
  p\,=\,(\,p_1:p_2:\cdots:p_{g-1}:1\,) \quad
 \hbox{with $\,{\rm val}(p_j)>0\,$ for $\,j=1,\ldots,g-1$.}
 $$
 
 Fix  $f\in\R[x_1,\ldots,x_g]$ with $f(p_e)\neq0$ but $f\cdot x_1x_2\in I_0$ and $f\cdot x_j\in I_0$ for $j=3,\ldots,g-1$. We could take $f$ to be a product of linear forms corresponding to hyperplanes 
  through all lines of $X_0$ other than $L_1$ and $L_2$. Moreover, we have ${\rm val}(f(p)) = 0$ since
  $f(p_e) \neq 0$.
  
  We can choose polynomials
  $y_{12},y_3,\ldots,y_g\in A[x_1,\ldots,x_g]$ such that 
 $$ fx_1x_2+\epsilon y_{12}\in I_\epsilon \quad {\rm and} \quad fx_j+\epsilon y_{j}\in I_\epsilon
 \,\,\, {\rm for} \,\,j=3,\ldots,g-1.$$
  If ${\rm val}(p_j)<1$ for some $j=3,\ldots,g-1$, then $(fx_j+\epsilon y_{j})(p)$ 
  has valuation equal to ${\rm val}(p_j)<1$, which would imply 
  $(fx_j+\epsilon y_{j})(p) \not= 0$.
   Thus ${\rm val}(p_j)\geq 1$ for all $j=3,\ldots,g-1$. This implies 
 \begin{equation} \label{eq:jacobian}
    \left(\frac{\partial(fx_j+\epsilon y_j)}{\partial x_i}(p)\right)_{\!3\leq i,j\leq g-1}
  = \quad f(p)\cdot I\,+\,\epsilon M,
 \end{equation}
 where $I$ is the identity matrix of size $g-4$ and $M$ is a square matrix with entries in $A$.
    Note that the matrix on the right-hand side of (\ref{eq:jacobian}) is invertible over $A$ because ${\rm val}(f(p))=0$. 
 
 We claim that ${\rm val}(p_j)\geq 1$ for $j=1,2$. This will be proved by contradiction.
 Suppose that ${\rm val}(p_j) < 1$ for $j\in \{1,2\}$.
 Since $p$ is a singular point of $Y$, the Jacobian matrix of the $g-2$ polynomials
 $fx_1x_2+\epsilon y_{12}, fx_3+\epsilon y_3,\ldots,fx_{g-1}+\epsilon y_{g-1} \in I_\epsilon$
 has rank at most $g-3$ at $p$.
 
 Consider the $(g-2)\times(g-2)$ minor given by partials with respect to $x_j, x_3,\dots, x_{g-1}$. The expansion of its determinant contains the term $f(p)^{g-2} \cdot p_j$. This term comes from
  $$\frac{\partial(fx_1x_2+\epsilon y_{12})}{\partial x_j}(p)$$
 and the matrix entries with valuation $0$ in (\ref{eq:jacobian}). All other terms have strictly bigger valuation, so they cannot cancel it. Therefore, the determinant is non-zero, which contradicts the fact that the Jacobian has rank at most $g-3$ at $p$.
 So, we indeed have ${\rm val}(p_j)\geq 1$ for $j=1,2$. 
 
 These two valuation inequalities imply ${\rm val}(y_{12}(p))\geq 1$ because
 $\,  (fx_1x_2+\epsilon y_{12})(p)=0$.
  This in turn implies the inequality ${\rm val}(y_{12}(p_e))\geq1$.
  Therefore, the constant term $y_{12}^0$ of
     $y_{12}$ with respect to $\epsilon$ vanishes at $p_e$. 
  We conclude that the tangent vector $\eta \in T_{x_0}(H_g)$ satisfies
 \begin{equation*}
 \eta(fx_1x_2)(p_e) \, = \, y_{12}^0(p_e) \,=\, 0 . 
 \end{equation*}
 This means that $\eta$ lies in the hyperplane $h_e$, which
  is a contradiction to our assumption.
  \end{proof}

We close this section by 
relating our hyperplane arrangements
to the Introduction.

\begin{example}
We examine the discriminants for the  complete intersections in
Examples \ref{ex:g3}, \ref{ex:g4} and  \ref{ex:g5}.
The lowest order term in each discriminant reveals
the hyperplane arrangement $\{ h_e \}_{e \in G}$.
For instance, the discriminant for ternary quartics is
a polynomial $\Delta$ of degree $27$ in the $15$ coefficients.
When evaluated at (\ref{eq:genus3}),  it is a polynomial
of degree $27$ in~$\epsilon$, namely
$$ \Delta(f) \cdot \epsilon^{27} \,+\, \cdots \,+\, \bigl(\,\prod_e h_e \,\bigr) \cdot \epsilon^6. $$
For a quartic $f$ and a cubic $g$ in $\PP^3$, the discriminant $\Delta$
has degree $77 $, and we obtain
$$ \Delta(f,g) \cdot \epsilon^{77} \,+\, \cdots \,+\, \bigl(\,\prod_e h_e \,\bigr) \cdot \epsilon^9. $$
Finally, for genus $g=5$, we consider the discriminant of three quadrics
$f,g,h$ in $\PP^4$, which has degree $120$ in the $45$ coefficients,
and the specialization to our family  (\ref{eq:genus5b}) has the form
$$ \Delta(f,g,h) \cdot \epsilon^{120} \,+\, \cdots \,+\, \bigl(\,\prod_e h_e \,\bigr) \cdot \epsilon^{12}. $$
Here $\prod_e h_e$ is the product of the $12$ linear forms in (\ref{eq:12ineqs}).
In each case, the degree of $\Delta$ was derived using methods in
 \cite{GKZ}, and we used computer algebra to verify the formulas above.
\end{example}

\section{Constructing MM-Curves}
\label{sec5}

We saw in Theorem \ref{thm:makeMumford}
that $Y$ is a smooth curve whenever the corresponding tangent vector $\eta$ lies in the complement of the hyperplane arrangement that consists of the $3g-3$ hyperplanes $h_e$ for $e$ an edge of $G$.  The $3g-3$ hyperplanes are in general position, by 
Lemma~\ref{lem:arrangement}. Hence our arrangement divides the tangent space
to the Hilbert scheme into $2^{3g-3}$ open cones. 

We next describe a bijection between these open cones and the set of edge pairings $\rho$ of $G$.
This uses  the notation from
 Lemma \ref{lem:welldefined}. Namely, $e = \{v_1,v_2\}$ denotes
 an edge of $G$ and $p_{e}$ the corresponding node of
$X_0 = C_G$.
     Let $L_1, L_2\subset X_0$ be the  lines that intersect at $p_e$ and $W_e\subset \PP^{g-1}$ the plane spanned by $L_1$ and $L_2$. Let $l_1,l_2\in\R[x_1,\ldots,x_g]$ two linear forms such that 
$V(l_i) \cap W_e = L_i$     
     for $i=1,2$. Finally, let $f$ be a polynomial such that $F=f\cdot l_1\cdot l_2\in I_0$ and $f(p_e)\neq0$. 
We denote by $e_1,e_1'$ and $e_2,e_2'$ the other edges of $G$ adjacent to $v_1$ and $v_2$. 

We choose vectors $P_{e_1},P_{e_1'},P_{e_2},P_{e_2'} \in \R^g$ representing the nodes $p_{e_1},p_{e_1'},p_{e_2},p_{e_2'}$~such~that
\begin{equation}
\label{eq:bigP1}
 P_e \,\,= \,\,P_{e_1}+P_{e_1'}\\
\,\,=\,\,P_{e_2}+P_{e_2'}. 
\end{equation}
Finally, we fix the following small positive real number:
\begin{equation}
\label{eq:bigP2}
 t_0 \,\,\,:=\,\,\,\min \bigl\{\,t>0\,\,:\, \,f(P_e+t(P_{e_1}\!+\!P_{e_2}))=0 \,\,\,{\rm or} \,\,\,
   f(P_e+t(P_{e_1'}\!+\!P_{e_2'}))=0 \,\bigr\}.
\end{equation}
With these conventions, we can now specify an edge pairing.
Namely, we define
\begin{equation}
\label{eq:bigP3}
 \rho(e)\,\,=\,\, \begin{cases} 
\, \{\{e_1,e_2\},\{e_1',e_2'\}\} & {\rm if}\,\,\, \eta(F)(P_e)\cdot F(P_e+t_0(P_{e_1}\!+\!P_{e_2}))\,<\,0 , \\
\, \{\{e_1,e_2'\},\{e_1',e_2\}\} &  {\rm otherwise}. \end{cases}
\end{equation}

We are now equipped to prove our main theorem which was
stated in the Introduction.

\begin{proof}[Proof of Theorem \ref{thm:main}]
Any MM-curve in $\PP^{g-1}$ with special fiber $C_G$ 
induces a double cover of the graph $G$
by $g+1$ cycles. By Theorem \ref{thm:main34},
this does not exist unless $G$ is planar.
We now fix a graph $G$ that is planar. Let
$x_0$ be the point corresponding to $C_G$ in $H_g$.
Consider the hyperplane arrangement in 
$T_{x_0}(H_g) \simeq \RR^{g^2-1} \times \RR^{3g-3}$
that was constructed in Section \ref{sec4}.

This hyperplane arrangement divides the tangent space into $2^{3g-3}$ open convex cones
$\RR^{g^2-1} \times \RR^{3g-3}_{> 0}$.
Namely, there is one cone for each edge pairing $\rho$ of $G$ as described above.
Let $\rho$ be the unique edge pairing such that the graph
$G_\rho$ in  Theorem \ref{thm:main34} has $g+1$~cycles.

Every tangent vector $\eta$ in the cone indexed by $\rho$
lifts to a smooth Mumford curve $Y$ with special fiber $C_G$.
This is the content of Theorem \ref{thm:makeMumford}. The locus of such curves $Y$ 
is a mixed semialgebraic set of full dimension $g^2+3g-4$
in the main component of the Hilbert scheme $H_g$.
Hence its image in the moduli space $\mathcal{M}_g$
also has the full dimension $3g-3$.

In order to prove Theorem~\ref{thm:main}, it remains to show that
the Mumford curve $Y$ is maximal. In other words, 
we must show that the curve $Y$ over $R$ constructed in Section \ref{sec4}
matches the construction of the graph $G_\rho$ in Section \ref{sec3}.
This is what we shall do in  Proposition~\ref{prop:match}.
\end{proof}


\begin{proposition} \label{prop:match}
 The number of semialgebraically connected components of $Y(R)$ is the number of
 cycles in the graph $G_{\rho}$. If
 $\,G_\rho$ consists of $\,g+1$ cycles then  $\,Y$ is an MM-curve.
\end{proposition}

\begin{proof}
 For $i=1,2$, let $L_{ee_i}$ be the open segment in the circle $L_i$
 that is bounded by $P_e$ and $P_{e_i}$ and does not contain $P_{e_i'}$.
 We similarly define the segment $L_{ee_i'}$ in $L_i$. By construction, 
  \begin{equation}\label{eq:reps} \!\!\!
  L_{ee_i}\,=\,\bigl\{[\lambda P_e+\mu P_{e_i}]\,:\, \lambda,\mu>0 \bigr\}
  \,\,\, {\rm and} \,\,\,
  L_{ee_i'}\,=\,\bigl\{[\lambda P_e+\mu P_{e_i'}]\,:\, \lambda,\mu>0 \bigr\}
  \,\,\, {\rm in} \,\,\, \PP^{g-1}(\RR).
 \end{equation}
 Let $\varphi$ be the map in  (\ref{eq:varphi})
 from the smooth curve to the graph curve.
 As in Section \ref{sec3}, we write $L_{ee_i}'$ for the closure of the preimage of $L_{ee_i}$ under $\varphi$ and  similarly for  $L_{ee_i'}$. We need to prove:
 \begin{equation}\label{eq:negi}
 \hbox{If $ \,\eta(F)(P_e)\cdot F(P_e+t_0(P_{e_1} \!+\!P_{e_2}))<0\,$
then $\,L_{ee_1}' \cup L_{ee_2}'\,$ is connected.}
\end{equation}
 Equivalently, we claim that 
$L'_{ee_1}\cup L'_{ee_2}$ and
   $L'_{ee_1'}\cup L'_{ee_2'}$ are disjoint.
These are the two pieces of the curve in (\ref{eq:preimages}).
Then (\ref{eq:negi}) ensures that our inequalities identify the correct edge pairing~$\rho$.

   We first note that multiplying $f$, $l_1$ and $l_2$ by a nonzero constant 
   does not change the sign on the left in (\ref{eq:negi}). Hence we can assume without loss of generality that 
   \begin{equation}
   \label{eq:wlog}
   f(P_e)>0 \, ,\,\,\, l_1(P_{e_2})>0\,\,\, {\rm and} \,\,\, l_2(P_{e_1})>0. 
   \end{equation}
     The linear form $l_1+l_2$ is positive on $L_{ee_i}$ but negative on $L_{ee_i'}$ for $i=1,2$.
     Here it is essential to represent points in $\PP^{g-1}(\RR)$ by the
      vectors in $\RR^g$ that are shown in
      (\ref{eq:reps}). Because the Hausdorff limit of $\varphi_t(L_{ee_i}')$ is the closure of $L_{ee_i}$, and similarly for $L_{ee_i'}$,  it suffices to prove that the line $V(l_1+l_2)$ does not intersect $X_t(\R)$ in a neighborhood of $P_e$ for small  $t>0$. 
 
  Our assumption on $t_0$ together with (\ref{eq:wlog})
    implies  $F(P_e+t_0(P_{e_1}+P_{e_2})) > 0$.
This, together with the hypothesis in (\ref{eq:negi}),  yields $\,\eta(F)(P_e)<0$. 
The ideal $I_\epsilon$ of the $R$-curve $Y$ contains 
\begin{equation}
\label{eq:Fplus2terms}
 \quad F\,+\,\epsilon \cdot \eta(F)\,+\,\epsilon^2 \cdot h
\quad {\rm where} \,\,\, h\in A[x_1,\ldots,x_g]. 
\end{equation}
The restriction of this polynomial to the line $V(l_1+l_2)$
is equal to $\,  -f\cdot l_1^2+\epsilon \cdot \eta(F)+\epsilon^2\cdot h$.
This is strictly negative near $P_e$ because $f(P_e)>0$ and $\eta(F)(P_e)<0$.
Since (\ref{eq:Fplus2terms}) vanishes~on~$Y$, we conclude that
$V(l_1+l_2) \cap Y(R) = \emptyset $ in a neighborhood of $P_e$ and thus also $V(l_1+l_2) \cap X_t(\R) = \emptyset $ in a neighborhood of $P_e$ for small enough $t>0$. This completes the proof.
\end{proof}

\begin{remark} \label{rem:nottheonly}
   The number of  circles of $X(\R)$, here denoted $r = r(X)$, is not the only discrete invariant of a smooth real genus $g$ curve $X$. The other invariant, denoted $a=a(X)\in\{0,1\}$, is the parity of the number of connected components of $X(\CC)\backslash X(\R)$ modulo $2$. The relation of $a$ to the number of circles $r$ and the genus $g$ of $X$ is given by the following two facts:
   \begin{enumerate}
       \item[(1)] If $r=g+1$, i.e. $X$ is an M-curve, then $a=0$. \vspace{-0.18cm}
       \item[(2)] If $a=0$, then $r$ is congruent to $g+1$ modulo $2$.
   \end{enumerate}
   We sketch, without giving proofs, how to determine the invariant a from our combinatorial data. Recall from the proof of Theorem \ref{thm:main34} that from an edge pairing $\rho$ we obtain a closed surface $S$ by gluing closed disks whose boundaries are the cycles of $G_\rho$ along edges from the same fiber of $G_\rho\to G$. Then $a=0$ if $S$ is orientable and $a=1$ otherwise. Note that this is consistent with
the   statements (1) and (2). Indeed, if $G_\rho$ has $r$ connected components, then the Euler characteristic of $S$ equals $r+1-g$. Since every closed surface of Euler characteristic $2$ is orientable, we have $a=0$ whenever $r=g+1$. Furthermore, every orientable closed surface has even Euler characteristic.
Thus, if $a=0$ then $r+1-g$ must be divisible by $2$.
\end{remark}

\section{Computations}
\label{sec6}

The following informal algorithm summarizes the construction of MM-curves in Section~\ref{sec5}.

\begin{algo} \label{alg:five}
The {\bf input} is a $3$-connected trivalent planar graph $G$ of genus $g$.
The {\bf output} is an {\em adapted basis} of the tangent space $T_{x_0}(H_g)$,
where $x_0$ is the point of the Hilbert scheme $H_g$ given by the
graph curve $C_G$.
Here ``adapted'' refers to the cone in Theorem~\ref{thm:main}:
\begin{equation}
\label{eq:MMcone} \RR^{g^2-1} \,\times \,\RR^{3g-3}_{> 0}. 
\end{equation}
The first $g^2-1$ vectors in our basis span the subspace $\RR^{g^2-1}$.
They are obtained by applying the $g^2$ differential operators $x_i \frac{\partial}{\partial x_j}$
to the generators of the ideal $I \subset S$ of the curve $C_G$. This yields
$g^2$ vectors in  ${\rm Hom}_S(I, S/I)_0$.
They satisfy one linear relation, given by the Euler operator
$\sum_{i=1}^g x_i \frac{\partial}{\partial x_i}$. We therefore omit the
last vector, i.e. the one that arises from $\,x_g \frac{\partial}{\partial x_g}$.

The remaining $3g-3$ basis directions are unique modulo the
subspace $\RR^{g^2-1}$ which arises from ${\rm PGL}(g)$.
We seek one basis vector $\eta_e$ for each edge $e$ of $G$.
The $\eta_e$ are constructed to span the factor $\RR^{3g-3}_{> 0}$ in 
the MM-cone (\ref{eq:MMcone}). We discuss two methods for computing them.

The first method uses the deformed quadric in  (\ref{eq:Xt1}).
We compute the ideal of the curve $X_t$ in  (\ref{eq:Xt2}),
and we apply the operator $\frac{\partial}{\partial t}|_{t=0}$ to each
ideal generator. This resulting tangent vector represents
  the first-order deformation $\eta = \eta_e$ which smoothes the node $p_e$.
  The second method starts with any basis of 
  ${\rm Hom}_S(I, S/I)_0$
and extracts a basis modulo the $(g^2-1)$-subspace described above.
 For each edge $e$ of $G$,
we  compute the polynomial $F = f l_1 l_2 \in I$
and we require that $\eta(F)(P_{e'}) = 0$
for all edges $e'$ other than $e$. This is a 
system of $3g-4$ linear equations in $3g-3$ unknowns.
Its solution $\eta = \eta_e$ is unique up to a scalar multiple.

It remains to scale each vector $\eta_e$ so that it 
actually lies in the closure of our cone $\RR^{3g-3}_{> 0}$.
We do this by computing the objects in (\ref{eq:bigP1}) and (\ref{eq:bigP2}), and 
we then scale $\eta_e$  so that
\begin{equation}
\label{eq:correctsign}
 \eta_e(F)(P_e) \,\, = \,\, - \,{\rm sign}\bigl( F(P_e + t_0 (P_{e_1} + P_{e_2})) \bigr)\,\,
\in \,\,\{-1,+1\}. 
\end{equation}
By (\ref{eq:bigP3}) and (\ref{eq:negi}),
this is the correct choice for the unique edge pairing that yields MM-curves.
The resulting basis of 
$T_{x_0}(H_g) = {\rm Hom}_S(I, S/I)_0$
is essentially unique. It is adapted to~(\ref{eq:MMcone}).
\end{algo}

We next show how Algorithm \ref{alg:five}
is used to create MM-curves from planar graphs.
Going well beyond the examples in the Introduction, our task now is to deform
graph curves $C_G$ that are not complete intersections in $\PP^{g-1}$.
The smallest such graph $G$ has genus five.

\begin{example}[$g=5$] \label{ex:G81}
We fix the edge graph $G$ of the simple $3$-polytope with eight vertices
shown on the left in Figure \ref{fig:G81b}. This is the
graph with the label \ {\tt 4.16.a}  \ in \cite[Figure 1~(d)]{GHSV}.
On the right in Figure \ref{fig:G81b} we see the graph 
$G_\rho$ with  six cycles that double-covers $G$ as in (\ref{eq:2to1map}).

Using  $(a:b:c:d:e)$ for the coordinates on $\PP^4$,
the graph curve $C_G$ has the radical ideal
\begin{equation}
\label{eq:G81ideal}
 I  \, = \,
\bigl\langle
(c +d)(c+d-e), \,a (c+d-e),\,a(b+d),\,(d-e)(ae+bc+bd),\,b(b+d)(d-e)
\bigr\rangle \, . \quad
\end{equation}
The five minimal generators of $I$ are the $4 \times 4$ Pfaffians in the
skewsymmetric $5 \times 5$ matrix
\begin{equation}
\label{eq:55skew}
\begin{footnotesize} \begin{bmatrix}
          0    &    0 &     0     &      b+d   &    c+d-e  \\
          0    &    0 &     a     &     -c-d   &      0    \\
          0    &   -a &     0     &       0    &   b(d-e) \\
        -b-d   &   c+d &    0     &       0    &   e(d-e) \\
        -c-d+e &   0 &   -b(d-e) &   -e(d-e) &      0   
\end{bmatrix}. \end{footnotesize}
\end{equation}
This form of the syzygies shows that $C_G$ lies in the closure of the trigonal locus of $\mathcal{M}_5$.

The ideals of the lines in $C_G$ are the associated primes of $I$.
We label them as in Figure~\ref{fig:G81b}:
$$
\begin{matrix}
1: \langle   b + d, c + d, e \rangle &
2:  \langle  b + e, c, d - e \rangle &
3: \langle  a \! - \! d,  b \!+\! d,  c \! +\! d \! -\! e \rangle &
4: \langle a,b, c + d - e \rangle \\
5: \langle a,c, d - e \rangle & 
6: \langle a, c + e, d - e \rangle &
7: \langle a, b, c + d \rangle &
8: \langle  a,b + d, c + d  \rangle
\end{matrix}
$$
A planar drawing of the graph curve $C_G$ is shown on the left in
Figure \ref{fig:G81d}.  The $12$ nodes  are
$$ \begin{matrix}
p_{12} =  (1\!:\!0 \! :\! 0\! :\!0\!: \!0) & 
p_{13} = (1{:} {-}1{:} {-}1{:} 1{:} 0) & 
p_{23} =  (1{:} {-}1{:} 0{:} 1{:} 1) &
p_{25} = (0{:} {-}1{:} 0{:} 1{:} 1) \\
p_{34} = (0{:} 0{:} {-}1{:} 0{:} {-}1) &
p_{45} = (0\!:\! 0\!:\! 0\!:\! 1\!:\! 1) &
p_{56} = (0\!:\! 1\!:\! 0\!:\! 0\!:\! 0) & 
p_{47} = (0{:} 0{:} {-}1{:} 1{:} 0) \\ 
p_{67} =  (0\!:\! 0 \!:\! {-}1 \!: \! 1\!: \! 1) &
p_{68} = (0{:} {-}1{:} {-}1{:} 1{:} 1) & 
p_{78} = (0{:} 0{:} 0{:} 0{:} {-}1) & 
p_{18} = (0{:} {-}1{:} {-}1{:} 1{:} 0).
\end{matrix}
$$

\begin{figure}[h] \vspace{-0.21in}
$$ \!\!\!\!\!\!\!\!\!\!
\includegraphics[width = 8.9cm]{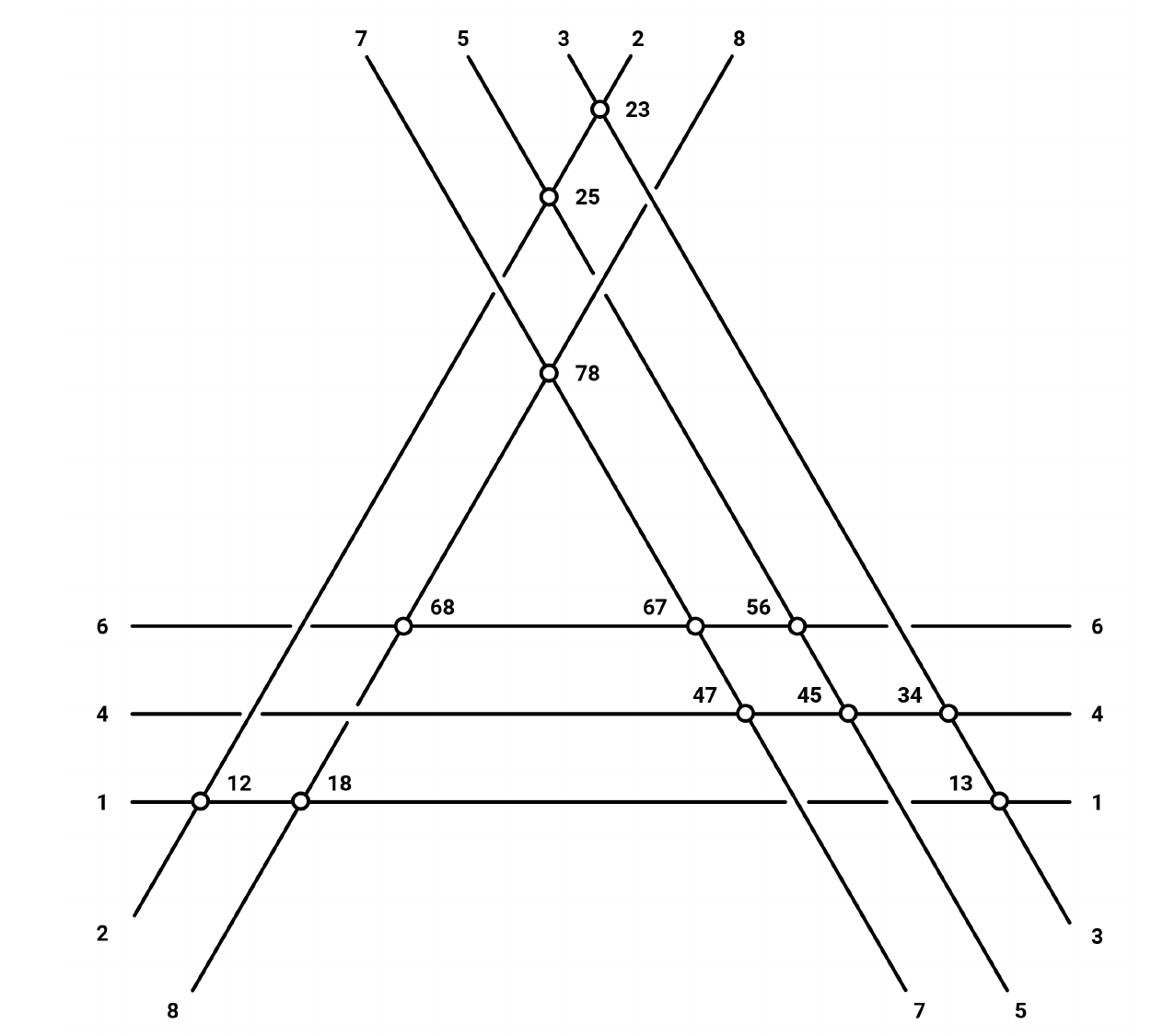} \,\,\,
\includegraphics[width = 8.9cm]{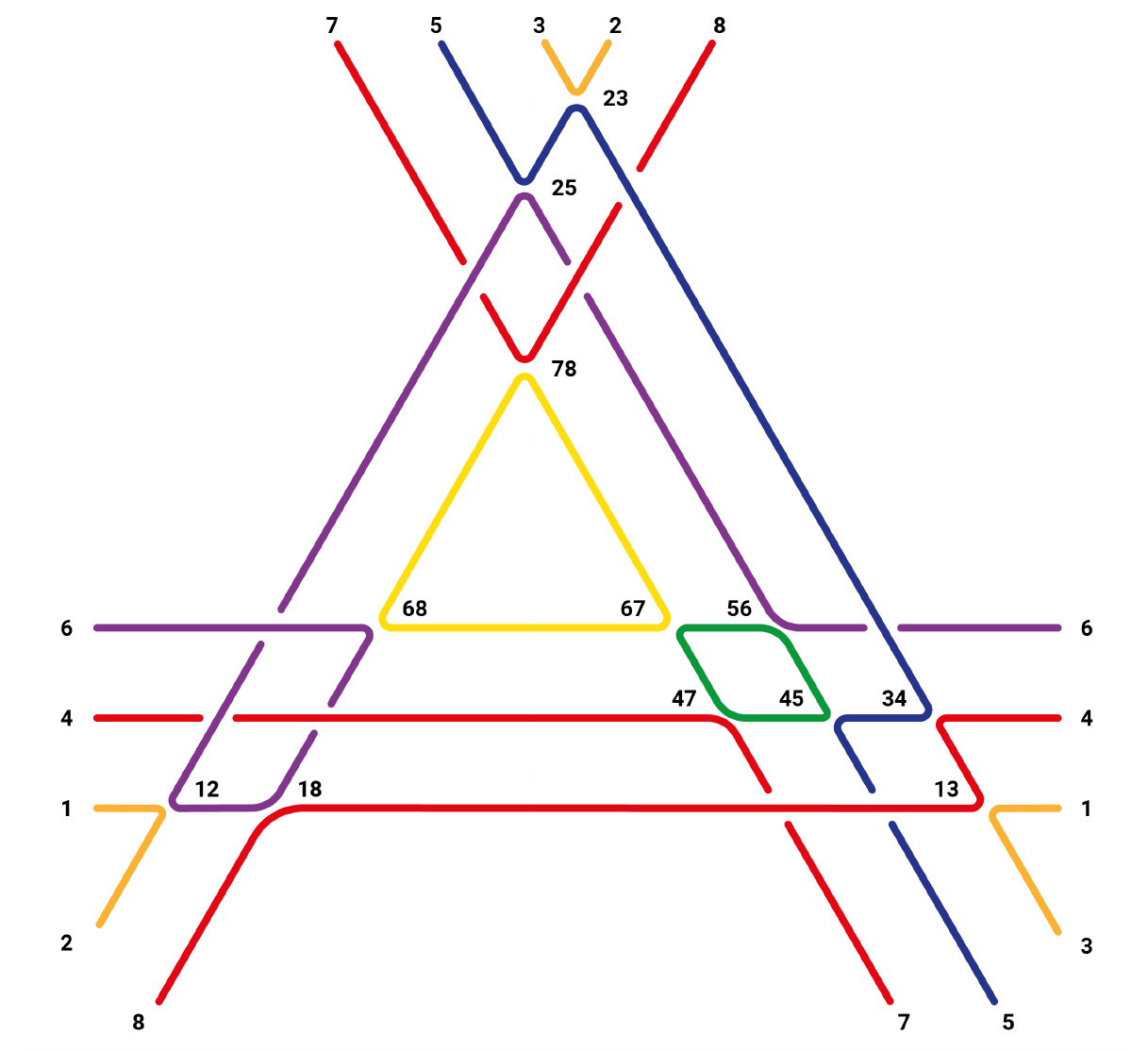} \vspace{-0.17in}
$$
\caption{A genus five graph curve and its deformation to an MM-curve with six ovals.}
\label{fig:G81d}
\end{figure}

Elements of the tangent space ${\rm Hom}(I,S/I)_0 \simeq \RR^{36}$
are represented by vectors in $S^5$. Their coordinates are the
images of the five generators in (\ref{eq:G81ideal}). Here $S = \RR[a,b,c,d,e]$.
The first $24$ elements in the adapted basis are obtained by applying the
Lie algebra of ${\rm PGL}(5)$:
$$ 
 \bigl(0, a(c\!+\!d\!-\!e), a(b\!+\!d), a(d\!-\!e)e, 0 \bigr),\,
\bigl(0, b(c\!+\!d\!-\!e), b(b\!+\!d), b(d\!-\!e)e, 0 \bigr),\,\ldots
$$
The other $12$ basis vectors are computed by one of the methods in Algorithm~\ref{alg:five}.
We find
$$ \! 
\begin{small}
\begin{matrix}
\eta_{12} = (0, 0, 0, -(a-d)^2a, 0) & 
\eta_{13} = (\,0,\, 0,\, 0,\, -a(d-e)^2, 0\,) \\
\eta_{23} \,=\,  (\,0,\, 0,\, 0, \,-ae^2,\, 0\,) &
\eta_{25} =  (\,0,\, 0,\, ae{+}bc{+}bd,\, 0,\, 0\,) \\
\eta_{34} = (0, 0, (d{-}e)(c{+}d), 0, (d{-}e)ce) &
\eta_{45} = (0, (c{+}d{-}e)e, (e{-}c{-}d)(b{+}d), (a{-}c{-}e)e^2, -(c{+}e)(b{+}d)e) \\
\eta_{56} \,=\, (\,(b+d)b,\, 0, \,0, \,0, \,0 \,) &
\eta_{47} = ((b+d)(d-e), 0, 0, 0, 0) \\
\eta_{67} =  (0, 0, (c{+}d)(b{+}d), 0, (b{+}d)ce) &
\eta_{68} \,=\, (\,0,\, 0,\, ae,\, 0,\, -bce\, ) \\
\eta_{78} = (0, 0, 0, 0, (d\!-\!e)^2(c\!+\!d\!-\!e)) &
\eta_{18} = ( 0, (d-e)(a+b), 0, -(d-e)^2 (a+b), 0).
\end{matrix}
\end{small}
$$
For any  $\lambda_{ij} > 0$, 
the following first-order deformation lifts to the equations of an MM-curve:
\begin{equation}
\label{eq:goodluck}
\hbox{the five generators in (\ref{eq:G81ideal}) }\,\, + \,\,
\epsilon \cdot \bigl(\lambda_{12} \eta_{12} + \lambda_{13} \eta_{13} + \cdots + \lambda_{78} \eta_{78}
+ \lambda_{18} \eta_{18} \bigr).
\end{equation}
A sketch of the six ovals of such MM-curve near $C_G$ is shown on the right in Figure \ref{fig:G81d}. We note that our diagram is similar to
 the curve from soliton data shown in
 \cite[Figure 16]{AG19}.
 \end{example}

We implemented Algorithm \ref{alg:five} in the computer 
algebra system {\tt Macaulay2} \cite{M2}.  
The code is made available, together with illustrations for
$g=6,7,8$, at our MathRepo page
$$ \hbox{\url{https://mathrepo.mis.mpg.de/mmcurves/}}. $$
The input for our {\tt Macaulay2} code is the graph $G$.
We write $x_0,x_1,\ldots,x_{g-1}$
for the variables in the polynomial ring $S$ corresponding to $\PP^{g-1}$.
The program computes the $2g-2$ associated primes of the ideal of
the curve $C_G$.
Their labels  $0,1,2,\ldots$ match the vertex labels of $G$.

We next input the desired edge pairing $\rho$. 
This is done by specifying one of the two pairs in $\rho(e)$, say
$\{e_i,e_j\}$ in the notation of (\ref{eq:edgepairing}). The other pair is determined.
In particular, for all planar graphs  $G$, we specify the edge pairing given by the
$g+1$ regions, as in Figure \ref{fig:G81b}.

The program loops over all edges $e$ of $G$ and computes the ideal of the curve $X_t$ in  (\ref{eq:Xt2}).
The computation is carried out over the ring $D$ of dual numbers.
The ideal generators for $X_t$ give rise to the corresponding tangent vector $\eta$.
Namely, for any generator $f$ in the ideal of  $X_0=C_G$, we compute $\eta(f)$
as the coefficient of $\epsilon$ in the corresponding ideal generator for $X_t$.
The correct sign is chosen according to  (\ref{eq:correctsign}).
The program then outputs all vectors~$\eta_e$. 
Alternatively, if only one
interior point in the cone $\RR^{g^2-1} \times \RR^{3g-3}_{> 0}$ is desired,
then it outputs the sum of the vectors $\eta_e$.
This first-order deformation lifts to an MM-curve near $C_G$.

\begin{figure}[h] \vspace{-0.17in}
$$ \!\!\!\!\!\!\!\!\!\!
\includegraphics[width = 8.8cm]{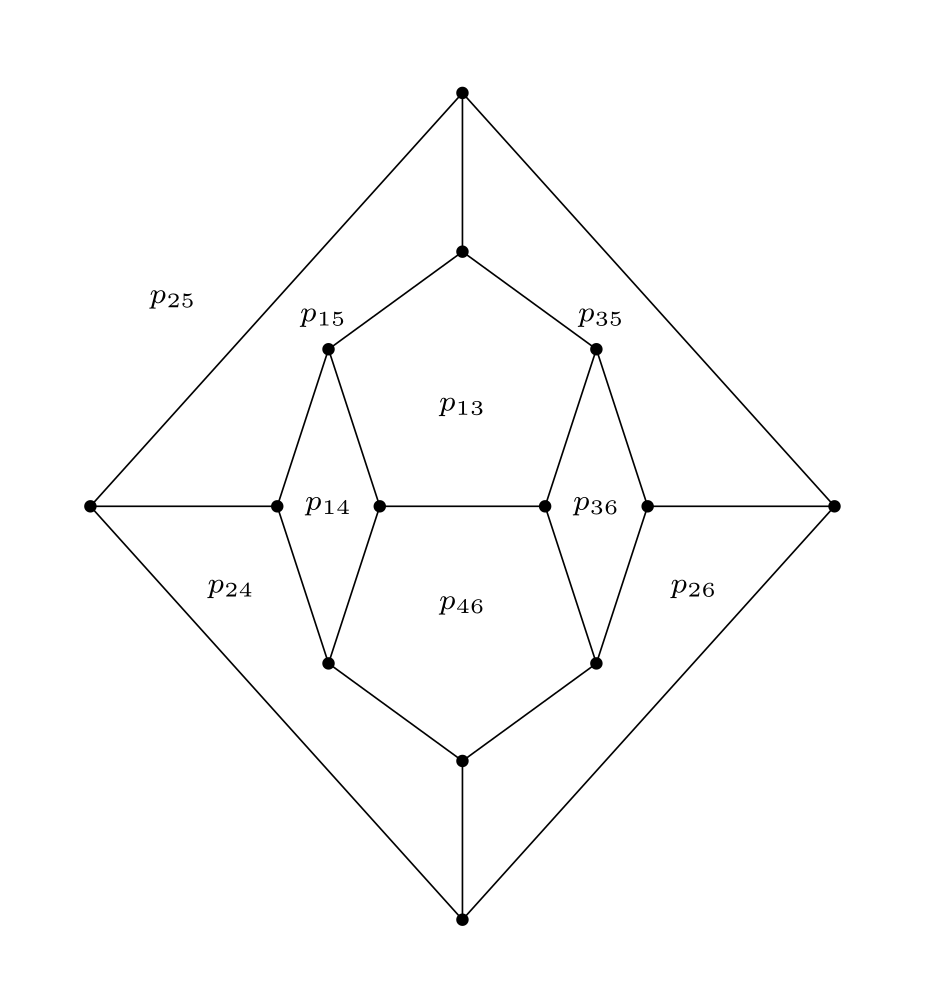} \vspace{-0.2in}
$$
\caption{The associahedron
has $14$ vertices, $21$ edges and $9$ facets, labeled $p_{13},p_{14},\ldots,p_{46}$.}
\label{fig:assoc}
\end{figure}

\begin{example}[Associahedron, $g=8$] \label{ex:asso}
As seen in \cite[Example 5.4]{GHSV},
the following monomial ideal in nine variables $p_{ij}$ 
is an initial ideal of the Pl\"ucker ideal of the Grassmannian $\textrm{Gr}(2,6)$:
\begin{equation} \label{eq:15gens}
 \begin{matrix} I \,\,=\,\,
\langle \, p_{13} p_{24}\,,\,p_{13} p_{25}\,,\,p_{13} p_{26}\,,\,p_{14} p_{25}\,,\,p_{14} p_{26}
\,,\,p_{15} p_{26} \,,\, p_{14} p_{35}\,,\, p_{14} p_{36}, \\ \qquad
        p_{15} p_{36}\,,\,\,p_{15} p_{46}\,,\,\,p_{24} p_{35}\,,\,\,p_{24} p_{36}
        \,,\,\,p_{25} p_{36}\,,\,\,p_{25} p_{46}\,,\,\,p_{35} p_{46} \,\rangle.
        \end{matrix}
\end{equation}
The corresponding simplicial complex is dual to the $3$-dimensional associahedron, see
Figure~\ref{fig:assoc}. The facets of this polytope are
labeled by Pl\"ucker coordinates. The ideal $I$ is generated by
pairs of disjoint facets.
To obtain a curve, we augment $I$ by a general linear form, e.g.
\begin{equation*}
    p_{13}- p_{14}+ p_{15}+ p_{24}- p_{25}+ p_{26}+ p_{35}+ p_{36}- p_{46}.
\end{equation*}
    The resulting variety   is the graph curve $C_G$ where $G$ is the graph in
     Figure~\ref{fig:assoc}.
    Our {\tt Macaulay2} code computes the
 tangent vectors  $\eta_e$ associated to the $21$ edges $e$, and it outputs their sum
    $$ \! 
\begin{footnotesize} \eta\, = \,
\begin{bmatrix}
p_{13} p_{14}-p_{14}^2+p_{14} p_{24}\\ -p_{15} p_{35}\\ -p_{13} p_{36}-p_{26} p_{36}-p_{36}^2\\ p_{15} p_{24}\\ -p_{14} p_{24}+p_{15} p_{24}+p_{24}^2-p_{24} p_{25}+p_{24} p_{26}\\ -p_{24} p_{25}-p_{25} p_{35}\\ -p_{14} p_{15}+p_{15}^2+p_{15} p_{24}+p_{24} p_{25}-p_{25}^2+p_{25} p_{26}+p_{15} p_{35}+p_{25} p_{35}\\ p_{13}^2{-}p_{13} p_{14}{+}p_{14} p_{15}{-}p_{15}^2
{-}p_{15} p_{24}{-}p_{24} p_{25}{+}p_{25}^2{-}p_{25} p_{26}{-}2 p_{15} p_{35}{-}p_{26} p_{35}
{-}p_{35}^2{+}p_{13} p_{36}{-}p_{35} p_{36}\\ -p_{15} p_{35}+p_{25} p_{35}-p_{26} p_{35}-p_{35}^2-p_{35} p_{36}\\ p_{13} p_{14}+p_{14} p_{24}\\ p_{24} p_{25}-p_{25}^2+p_{25} p_{35}\\ -p_{24} p_{26}+p_{25} p_{26}-p_{26}^2-p_{26} p_{35}-p_{26} p_{36}\\ -p_{26} p_{35}\\ p_{24} p_{26}\\ -p_{13} p_{36}-p_{26} p_{36}
\end{bmatrix}\! .
\end{footnotesize}
$$
The $i$th coordinate is the image under $\eta \in {\rm Hom}_S(I,S/I)_0$
of the $i$th generator of the ideal $I$.
Every deformation in direction $\eta$ lifts to an MM-curve 
in $\PP^7_R$ whose special fiber is $C_G$.
\end{example}

We emphasize that the output of Algorithm \ref{alg:five} 
is only a first-order deformation. While our theory guarantees that
this output lifts to an actual MM-curve in $\PP^{g-1}_R$,
at present we do not have a practical
algorithm for computing such liftings.
In special cases, ad hoc techniques can be used to
compute the desired lifting.
We now show such an MM-curve
for Example~\ref{ex:G81}.

\begin{example}[$g=5$, continued] 
\label{ex:G81more}
Let $G$ be the graph   in Figures \ref{fig:G81b} and \ref{fig:G81d}.
We claim that the following  equations in $a,b,c,d,e$ define
an MM-curve in $\PP^4(R)$ with
special fiber $C_G$:

\begin{equation}
\label{eq:limitgen6}
\begin{small}
\begin{matrix}
(c+d)(c+d-e) \,\,+\,\, \epsilon \cdot (b^2+2bd-be+d^2-de) \,\,
+\,\, \epsilon^2 \cdot a^2 & = & 0, \\
a (c+d-e) \,\,+\,\, \epsilon \cdot (ad-ae+bd-be+ce+de-e^2) \,\,-\,\,
\epsilon^2 \cdot e^2 & = & 0 ,\\
a(b+d) \,\,+\,\, \epsilon \cdot (ae+bc+bd+be+cd-ce+d^2) \phantom{dadadadada} & = & 0.
\end{matrix}
\end{small}
\end{equation}
We will show that for small $\epsilon  > 0$, this is a smooth canonical curve
with six ovals in $\PP^4(\RR)$. For $\epsilon = 0$,
the three equations define a union of three planes.
The correct limit  $x_0=\lim_{ \epsilon \to 0} x_\epsilon$ in the Hilbert scheme is 
the point corresponding to the graph curve $C_G$. This is verified by saturating
with respect to $\epsilon$.
The tangent vector $\eta$ corresponding to this flat family equals
$$ \eta \,\,= \,\,
\begin{footnotesize}
\begin{bmatrix}
b^2+2 b d+d^2-b e-d e\\ a d+b d-a e-b e+c e+d e-e^2\\ 
ae + b c+b d+be+ c d-c e+d^2
\\ 
-a^3{+}b^2 c{-}a d^2{-}2 b d^2{+}c d^2{+}d^3{+}b^2 e
{+}3 b d e{-}3 c d e{-}3 d^2 e{-}b e^2{+}c e^2{+}2 d e^2{-}e^3
\\ b^3-c d^2+b^2 e+a d e-b d e+2 c d e-d^2 e-a e^2+d e^2-e^3
\end{bmatrix}\! .
\end{footnotesize}
$$
We express this tangent vector as a linear combination of our adapted basis
in Example \ref{ex:G81}:
\begin{equation*}    \eta \,\,=\,\,
    \eta_{12}+ \eta_{13}+ 2\eta_{23}+ \eta_{25}+ \eta_{34}+ \eta_{45}+ \eta_{56}
    + \eta_{47}+ 2\eta_{67}+ \eta_{68}+ \eta_{78}+ \eta_{18}+ \ldots
\end{equation*}
Here we omitted the basis vectors obtained by applying the Lie algebra of ${\rm PGL}(5)$. 
Since all twelve $\eta_{ij}$ have positive coefficients, 
(\ref{eq:limitgen6}) is an MM-curve in $\PP^4(R)$ with special fiber $C_G$.

We note that the $\epsilon^2$ terms in~(\ref{eq:limitgen6}) are essential. Without them, 
(\ref{eq:limitgen6}) would be equal to
\begin{equation}
\label{eq:badguess}
\hbox{the three quadrics in (\ref{eq:G81ideal}) }\,\, + \,\,
\epsilon \cdot \bigl(\eta_{12} + \eta_{13} + \cdots + \eta_{78}
+ \eta_{18} \bigr) \, .
\end{equation}
This naive linear deformation
cuts out a singular curve.
This is not the lift of any tangent vector in the MM-cone. The tangent vector $\eta$ corresponding to our flat family (\ref{eq:badguess}) equals
\begin{equation*}    \eta \,\,=\,\,
    \eta_{25}+ \eta_{34}+ \eta_{56}
    + \eta_{47}+ \eta_{67}+ \eta_{68}+ \eta_{78}+ \eta_{18}+ \ldots.
\end{equation*}
This lies in the boundary of our cone since $\eta_{12}, \eta_{13}, \eta_{23}$ and $\eta_{45}$ appear with coefficient zero.
\end{example}

\bigskip

\bigskip

\noindent {\bf Acknowledgment.}
 We are grateful to
Gavril Farkas, Stefan Felsner, Alex Fink, Daniel Kral and Martin Ulirsch
for helping us by answering our mathematical questions. We also thank Jonathan Hauenstein and Ben Hollering for their assistance with computations.

                \bigskip
                \bigskip

                \footnotesize
                \noindent {\bf Authors' addresses:}

                \smallskip

                \noindent Mario Kummer, TU Dresden,
                		\hfill \url{mario.kummer@tu-dresden.de}

                \noindent  Bernd Sturmfels, MPI-MiS Leipzig
                \hfill \url{bernd@mis.mpg.de}

                \noindent Raluca Vlad, Brown University
                \hfill \url{raluca_vlad@brown.edu}


\bigskip


\begin{thebibliography}{10}
\begin{small}
                                \setlength{\itemsep}{-0.6mm}

\bibitem{AG19}
S.~Abenda and P.~Grinevich:
{\em Reducible M-curves for Le-networks in the totally-nonnegative Grassmannian and KP-II multiline solitons},  Selecta Mathematica {\bf 25} (2019) 43.

\bibitem{AG}
S.~Abenda and P.~Grinevich:
{\em Real regular KP divisors on M-curves and totally non-negative Grassmannians},
Letters in  Mathematical Physics {\bf 112} (2022) paper 115.

\bibitem{AGS}
X.~Allamigeon, S.~Gaubert and M.~Skomra:
{\em Tropical spectrahedra},
Discrete and Computational Geometry
{\bf 63} (2020) 507--548.

\bibitem{basuetal}
S.~Basu, R.~Pollack and M.-F.~Roy:
 {\em Algorithms in Real Algebraic Geometry},
 Algorithms Comput. Math., vol~{\bf 10}, Springer, Berlin, 2003.

\bibitem{bertini}
D.~Bates, J.~Hauenstein, A.~Sommese and C.~Wampler:
{\em Bertini: Software for Numerical Algebraic Geometry},
 available at  \url{bertini.nd.edu}, and with \ dx.doi.org/10.7274/R0H41PB5.
 
 
\bibitem{BE} D.~Bayer and D.~Eisenbud:
{\em Graph curves}, Advances in Mathematics
{\bf 86} (1991) 1--40.


\bibitem{rag}
J.~Bochnak, M.~Coste and M.-F.~Roy:
{\em Real Algebraic Geometry},
  {\em Ergeb. Math. Grenzgeb.}, vol~{\bf 36},
Springer, Berlin, 1998.

\bibitem{Bra}
P.~Bradley:
{\em Generalised diffusion on moduli spaces of p-adic Mumford curves},
p-Adic Numbers Ultrametric Anal.~Appl. {\bf 12} (2020) 73--89.

\bibitem{CS} M.~Chan and B.~Sturmfels:
{\em Elliptic curves in honeycomb form},
Tropical Geometry (Castro Urdiales 2011),
American Math.~Society,
Contemporary Mathematics {\bf 589} (2013) 87--107. 

\bibitem{Diestel} R.~Diestel: {\em Graph Theory},
Graduate Texts in Mathematics, vol {\bf 173}, Springer, Berlin, 2005.
  
\bibitem{geometryofschemes}
D.~Eisenbud and J.~Harris: {\em The Geometry of Schemes}, 
Graduate Texts in Mathematics, vol {\bf 197}, Springer, New York, 2000.


\bibitem{GHSV}
A.~Geiger, S.~Hashimoto, B.~Sturmfels and R.~Vlad:
{\em Self-dual matroids from canonical curves},
Experimental Mathematics (2024), published on-line.

 \bibitem{GKZ}
 I.~Gel'fand, M.~Kapranov and A.~Zelevinsky: {\em Discriminants, Resultants, and Multidimensional Determinants},  Birkh\"auser, Boston, 1994.

\bibitem{Ger}
L.~Gerritzen: {\em Zur analytischen Beschreibung des Raumes der Schottky-Mumford Kurven},
Mathematische Annalen {\bf 255} (1981) 259--271.

\bibitem{M2} D.~Grayson  and M.~Stillman:  Macaulay2, a software system
for research in algebraic geometry, available at
{\tt http://www.math.uiuc.edu/Macaulay2/}.

\bibitem{GJ}
T.~Gunn and P.~Jell:
{\em Construction of fully faithful tropicalizations for curves in ambient dimension 3},
{\tt arXiv:1912.02648}.

\bibitem{moduli}
J.~Harris and I.~Morrison: {\em Moduli of Curves}, 
Grad. Texts Math., vol {\bf 187}, Springer, New York,~1998.

\bibitem{hartshorne}
R.~Hartshorne: {\em Algebraic Geometry},
  Grad. Texts Math., vol {\bf 52}, Springer, Cham, 1983.
  
\bibitem{HH}
R.~Hartshorne and A. Hirschowitz:
{\em Smoothing algebraic space curves}, 
Algebraic geometry, Proc. Conf., Sitges (Barcelona) 1983, Lect. Notes Math. {\bf 1124}
(1985) 98--131.

 \bibitem{Hel}
P.~Helminck:  {\em Tropical Igusa invariants},
{\tt arXiv:1604.03987}.

\bibitem{Her}
F.~Herrlich: {\em The nonarchimedean extended Teichm\"uller space},
p-adic analysis, Proc. Int. Conf., Trento/Italy 1989, 
Lecture Notes in Mathematics {\bf 1454} (1990) 256--266.

\bibitem{Ich}
T.~Ichikawa: {\em
The universal Mumford curve, and its abelian differentials and periods 
in arithmetic formal geometry}, {\tt arXiv:2010.11517}.

\bibitem{Jell} P.~Jell:  {\em Constructing smooth and fully faithful tropicalizations for Mumford curves},
Selecta Mathematica {\bf 26} (2020) 60.

\bibitem{JSY} P.~Jell, C.~Scheiderer and J.~Yu:
{\em Real tropicalization and analytification of semialgebraic sets},
International Mathematics Research Notices, 2022, No. 2, (2022) 928-958.

\bibitem{MacLane1937} S.~Mac Lane:
  {\em A combinatorial condition for planar graphs},
  Fundam.~Math.~{\bf 28} (1937) 22--32.

\bibitem{MR}
R.~Morrison and Q.~Ren: {\em Algorithms for Mumford curves},
J.Symb.Comput.~{\bf 68} (2015) 259-284.
		 
\bibitem{MP}
M.~Mulase and M.~Penkava:
{\em Ribbon graphs, quadratic differentials on Riemann surfaces, and algebraic curves 
defined over $\overline{\QQ}$}, Asian Journal of Mathematics {\bf 2} (1998) 875--919.

\bibitem{Nic} J.~Nicaise:
{\em Geometric invariants for non-archimedean semialgebraic sets},
Algebraic geometry: Salt Lake City 2015, Proceedings. Part 2.
Providence, RI: American Mathematical Society; Cambridge, MA:
Clay Mathematics Institute. Proc. Symp. Pure Math. 97, {\bf 2}
 (2018) 389--403.

\bibitem{Rez} F.~Rezaee:
{\em Geometry of canonical genus four curves},
Proc.Lond.Math.Soc.~{\bf 128} (2024)  e12577.
 
 \bibitem{SS}
 M.~Sepp\"al\"a and R.~Silhol:
 {\em Moduli spaces for real algebraic curves and real abelian varieties},
 Mathematische Zeitschrift {\bf 201} (1989) 151--165.
 
 \bibitem{Stu} B.~Sturmfels:
 {\em Viro's theorem for complete intersections},
 Annali della Scuola Normale Superiore di Pisa {\bf 21} (1994) 377--386.

\bibitem{Viro} O.~Viro:
{\em Real algebraic plane curves: Constructions with controlled topology},
Leningrad Mathematical Journal {\bf 5} (1990) 1059--1134.

\end{small}
\end{thebibliography}
 \end{document}